\documentclass[11pt]{article}
\usepackage[utf8]{inputenc}
\usepackage{amsmath}   
\usepackage{amsfonts}  
\usepackage{amssymb}  
\usepackage{latexsym}  
\usepackage{graphicx}
\usepackage{epstopdf}
\usepackage[a4paper]{geometry}
\usepackage{color}

\usepackage{subfigure}
\usepackage{url}

\usepackage{authblk}

\usepackage{hyperref}

\usepackage{graphicx}

\graphicspath{{Figures/}}

\newtheorem{theorem}{Theorem}
\newtheorem{proposition}[theorem]{Proposition}

\newtheorem{lemma}[theorem]{Lemma}

\newtheorem{claim}{Claim}

\newcommand{\mb}[1]{\mathbb{#1}}

\def\hM{\hat{M}}

\def\Nbb{N_{\bullet\bullet}}
\def\Nbw{N_{\bullet\circ}}
\def\Nwb{N_{\circ\bullet}}
\def\Nww{N_{\circ\circ}}

\def\Sbw{S_{\bullet\circ}}

\def\Sww{S_{\circ\circ}}
\def\Sbb{S_{\bullet\bullet}}

\def\C2{C^{(2)}}

\def\Rb{R_{\bullet}} 
\def\Rw{R_{\circ}}
\def\rb{r_{\bullet}} 
\def\rw{r_{\circ}}

\def\Minf{M^{\infty}}

\def\cB{\mathcal{B}}
\def\cH{\mathcal{H}}
\def\cO{\mathcal{O}}
\def\cD{\mathcal{D}}

\def\cQ{\mathcal{Q}}
\def\cT{\mathcal{T}}
\def\cU{\mathcal{U}}
\def\cX{\mathcal{X}}
\def\cUbal{\cU^{\mathrm{Bal}}}
\def\cXbal{\cX^{\mathrm{Bal}}}

\def\cTri{\cT^{(3)}}
\def\cUbaltri{\cU^{\mathrm{Bal},3}}
\def\cHtri{\cH^{(3)}}

\def\cUh{\widehat{\cU}}

\def\zb{z_{\bullet}}
\def\zw{z_{\circ}}
\def\tb{t_{\bullet}}
\def\tw{t_{\circ}}
\def\zZ{\mathbb{Z}}
\def\Pi{P^{(i)}}

\def\Cbbi{C_{\bullet\bullet}^{(i)}}
\def\Cbwi{C_{\bullet\circ}^{(i)}}
\def\Cwwi{C_{\circ\circ}^{(i)}}

\newcounter{sclaim}
\newcounter{ssclaim}

\newenvironment{proof}{\noindent \setcounter{ssclaim}{0}\emph{Proof.}\ }{\hfill
    $\Box$\vspace{1em}}

{\refstepcounter{sclaim}\vspace{1ex}\noindent{\it  (\arabic{sclaim})  {#1}{}}\it}{\vspace{1ex}}

	{\noindent {}{#1}{}}{ This proves (\arabic{sclaim}).\vspace{1ex}}

        \title{A bijection for essentially 3-connected toroidal maps
          \thanks{This work was
      supported by the ANR grant GATO
      ANR-16-CE40-0009-01.}}

\author{Nicolas Bonichon\thanks{LaBRI UMR 5800, Universit\'e Bordeaux, 351 cours de la Lib\'eration 33405 Talence Cedex, France.  \texttt{bonichon@labri.fr}}\ ,  
\'Eric Fusy\thanks{LIGM UMR 8049, Universit\'e Gustave Eiffel,   
5 boulevard Descartes, 77454 Marne-la-Vallée Cedex 2, France. \texttt{eric.fusy@u-pem.fr}}\ ,
Benjamin
  L\'ev\^eque\thanks{G-SCOP UMR 5272, Universit\'e Grenoble Alpes, 46 avenue F\'elix Viallet 
38031 Grenoble Cedex 1, France. \texttt{benjamin.leveque@cnrs.fr}} 
}

\begin{document}
\maketitle
\begin{abstract}
We present a bijection for toroidal maps that are essentially $3$-connected ($3$-connected in the periodic planar  
representation). Our construction actually proceeds on certain closely related bipartite toroidal maps with all faces of degree $4$ except for a hexagonal root-face. 
We show that these maps are in bijection with certain well-characterized bipartite unicellular maps. Our bijection, closely related to the recent one by Bonichon and L\'ev\^eque for essentially 4-connected toroidal triangulations, can be seen as the toroidal counterpart of the one  developed in the planar case by Fusy, Poulalhon and Schaeffer, and it extends the one recently proposed by Fusy and L\'ev\^eque for essentially simple toroidal triangulations. Moreover, we show that rooted essentially $3$-connected toroidal maps can be decomposed into two pieces, a toroidal part that is treated by our bijection, and a planar part that is treated by the above-mentioned planar case bijection. 
 This yields a combinatorial derivation for the bivariate generating function of rooted essentially $3$-connected toroidal maps, counted by vertices and faces. 
\end{abstract}


\vspace{.3cm}

\section{Introduction}
Bijective combinatorics of planar maps is a very active research topic, with several applications such as scaling limit results~\cite{le2013uniqueness,miermont2013brownian,addario2017scaling,bettinelli2017compact}, efficient random generation~\cite{schaeffer1999random} and mesh encoding~\cite{PS03b,aleardi2008succinct}
(as well as giving combinatorial interpretations of beautiful counting formulas first discovered by Tutte~\cite{T62a,Tu63} 
in the 60's). In these bijections a map is typically encoded by a decorated tree structure.  
Bijections for planar maps with no girth nor connectivity\footnote{The girth of a graph $G$ is the length of a shortest cycle in $G$; and $G$ is called $k$-connected if one needs to delete at least $k$ vertices to disconnect it.}  conditions rely on geodesic labellings~\cite{Sc97,BoDiGu04}   
and can be extended to any fixed genus~\cite{CMS09,chapuy2009asymptotic,Lep18,chapuy2017bijection} (where the map is encoded by a decorated unicellular map). On the other hand, when a girth condition or connectivity condition is imposed the bijections typically rely (see~\cite{AlPo13,BF12} for general methodologies) on the existence of a certain `canonical' orientation with prescribed outdegree conditions, such as Schnyder orientations 
(introduced by Schnyder~\cite{S89} for simple planar triangulations, and later
generalized to 3-connected planar maps~\cite{Fe03} and to $d$-angulations of girth $d$~\cite{BF12,bernardi2012schnyder}), separating decompositions~\cite{de1995bipolar} or transversal structures~\cite{kant1997regular,Fu07b}; and it is not known how to specify such a canonical orientation in any fixed genus 
(we also mention the powerful approach by Bouttier and Guitter~\cite{bouttier2014irreducible} based on 
 decomposing so-called \emph{slices}; this method, which yields a unified combinatorial decomposition for 
irreducible maps with control on the girth and face-degrees, is yet to be extended to higher genus). 
 
However, it has recently appeared that in the particular case of genus
$1$ (toroidal maps) a notion of canonical orientation amenable to a
bijective construction can be found, with the nice property that the
outdegree condition on vertices is completely homogenous (which is not
the case in the planar case, where the outdegree conditions differ for
vertices incident to the outer face). This has been done for
essentially simple\footnote{For a toroidal map $M$, with $M^{\infty}$
  the corresponding infinite periodic planar representation, we say
  that $M$ is essentially simple if $M^{\infty}$ is simple, and more
  generally that $M$ has essential girth $d$ if $M^{\infty}$ has girth
  $d$, and we say that $M$ is essentially $k$-connected if
  $M^{\infty}$ is $k$-connected.} triangulations~\cite{DGL15} (based
on earlier work on toroidal Schnyder woods~\cite{GL14}, with graph drawing
applications) and more generally for maps with essential
girth $d\geq 1$ and root-face degree $d$~\cite{FL18}, and for
essentially 4-connected triangulations~\cite{BoLe18}. It has also been
shown in~\cite{FL18} that, similarly as in the planar
case~\cite{BF12}, such bijections (at least, the one in~\cite{FL18})
can be obtained by specializing a certain `meta'-bijection that is
derived from a construction due to Bernardi and Chapuy~\cite{BC11}.

In this article we apply this methodology to the 
family of essentially 3-connected toroidal maps, where our bijection can be seen as the natural toroidal counterpart of the bijection developed in~\cite{FuPoScL} between unrooted binary trees and so-called irreducible quadrangular dissections of the hexagon, which are closely related to  3-connected planar maps. Precisely, we describe in Section~\ref{sec:bij_psi} a bijection $\psi$ (based on local closure operations very similar to those used in the planar case construction~\cite{FuPoScL})  between a certain family of unicellular maps and a certain family $\cH$ of bipartite toroidal maps with one face of degree $6$ and all the other faces of degree $4$ (and having the property of being 
`essentially irreducible', see Section~\ref{sec:defs} for definitions). This construction actually bears a strong resemblance with the recent bijection for essentially 4-connected toroidal triangulations described in~\cite{BoLe18}
(this resemblance already appeared for the planar counterpart constructions given respectively
in~\cite{FuPoScL} and in~\cite{Fu07b}), and we conjecture that a unified bijection for `essentially irreducible' toroidal $d$-angulations should exist.   
 
In Section~\ref{sec:link} we then show that any 
rooted essentially $3$-connected toroidal map can be decomposed (via the associated bipartite quadrangulation) into two parts by cutting along a certain `maximal' hexagon: 
a planar part that can be treated by the planar case bijection~\cite{FuPoScL}, and a toroidal part that can be treated by our bijection $\psi$. We then obtain in Section~\ref{sec:counting}  
 a combinatorial derivation of the bivariate generating function of rooted essentially 3-connected toroidal maps counted by vertices and faces. 

Similarly as in the planar case~\cite{FuPoScL}, the inverse bijection
$\phi$, which starts from $\cH$  and is described in Section~\ref{sec:inverse}, relies on canonical 3-biorientations (edges are either simply directed or bidirected, every vertex has outdegree $3$), to which the above-mentioned
meta-bijection can be specialized. These biorientations are closely related to so-called balanced toroidal Schnyder orientations whose existence has been shown in~\cite{GL14,LevHDR}.  

\section{Preliminaries on maps}\label{sec:defs}
A \emph{map} $M$ of genus $g$ is given by the embedding (up to
orientation-preserving homeomorphism) of a connected graph $G$
(possibly with loops and multiple edges) on a compact orientable
surface $\Sigma$, such that all components of $\Sigma\backslash G$ are
homeomorphic to topogical disks, which are called the \emph{faces} of
$M$. The \emph{genus} of $M$ is the genus of the underlying surface
$\Sigma$; maps of genus $0$ and $1$ are called planar and toroidal,
respectively. Since the universal cover of the torus is the periodic
plane, a toroidal map $M$ can also be seen as an infinite periodic
planar map which we denote by $\Minf$, see Figure~\ref{fig:maps}(a)
(in our figures, toroidal maps are drawn on the flat torus
$\mathbb{R}^2/\mathbb{Z}^2$, i.e., the unit square where the opposite
sides are identified; and the drawing lifts to a periodic planar
drawing upon replicating the square to tile the whole plane).
   
A \emph{corner} of $M$ is the angular sector between two consecutive half-edges around a vertex. 
The \emph{degree} of a vertex or face of $M$ is the number of corners that are incident to it. 
A map is called \emph{rooted} if it has a marked corner, and is called \emph{face-rooted} if it has a marked
face. 
A \emph{triangulation} (resp. \emph{quadrangulation}) is a map with all faces of degree $3$
 (resp. of degree $4$). A map is called \emph{6-quadrangular} if it has one face of degree $6$
and all the other faces have degree $4$ (we will refer to the face of degree $6$ as the root-face). 

\begin{figure}
\center
\includegraphics[scale=0.77]{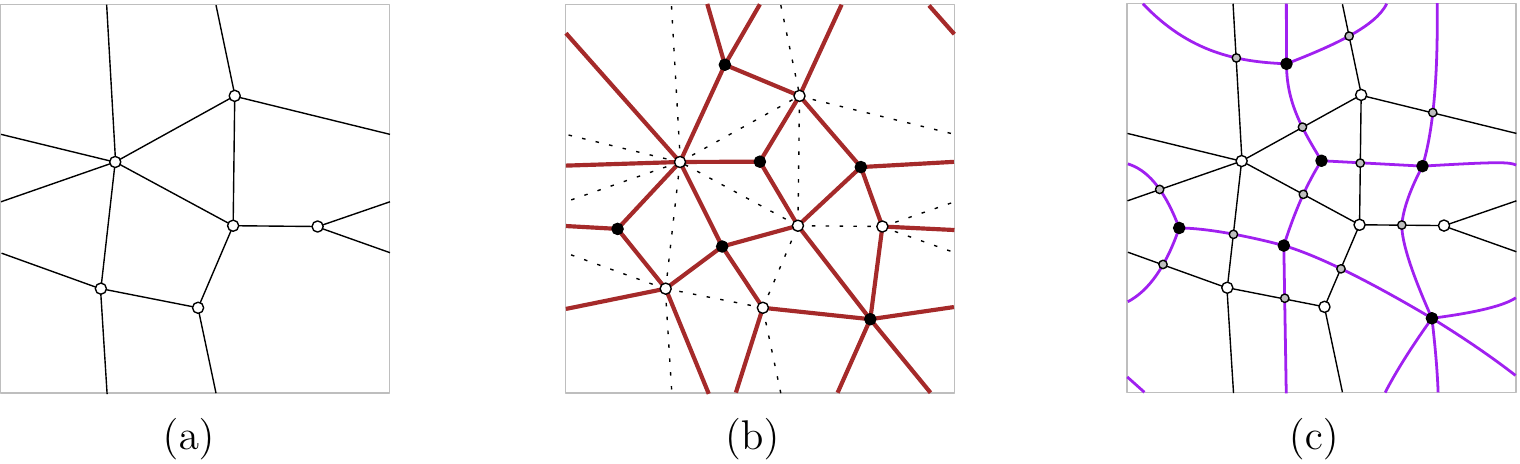}
\caption{(a) A toroidal map $M$. (b) The angular map of $M$. (c) The derived map $\hM$.}
\label{fig:maps}
\end{figure}

A graph or map is called \emph{bipartite} if its vertices are partitioned into black and white vertices, 
such that every edge connects a black vertex to a white vertex.
For $M$ a map (whose vertices are white), the \emph{angular map} of $M$ is the bipartite 
map $Q$ obtained by inserting a black vertex $v_f$
 inside each face $f$ of $M$, and connecting $v_f$ to the vertices of $M$ at every corner around $f$,
and then deleting all the edges of $M$ (see Figure~\ref{fig:maps}(b)). 
It is well known that 
 the mapping from $M$ to $Q$ is a bijection between maps of genus $g$ with $i$ vertices and $j$ faces, 
and bipartite quadrangulations of genus $g$ with $i$ white vertices and $j$ black vertices.
  
On the other hand, the \emph{dual} of $M$ is the map $M^*$ obtained
by inserting a vertex $v_f$ in each face $f$ of $M$, and then for each edge $e$ of $M$, with $f,f'$ the 
faces on each side of $e$, drawing an edge $e^*$ accross $e$ that connects $v_f$ to $v_{f'}$,
and finally deleting the vertices and edges of $M$. 
The \emph{derived map} of $M$ is the map $\hM$ obtained by superimposing $M$ with $M^*$ (see Figure~\ref{fig:maps}(c)). There are 3 types of vertices in $\hM$: the vertices of $M$ are called \emph{primal vertices}, the vertices of $M^*$
are called \emph{dual vertices}, and the vertices of degree $4$ at the intersection of each edge $e$ with its dual edge $e^*$, are called \emph{edge-vertices}. 
Note that $\hM$ is actually a quadrangulation. In fact, it is easy to see that it is 
the angular map of the angular map of $M$.  
Note that $M^*$, $Q$, and $\hM$ all have the same genus as $M$.

A (possibly infinite) map is called \emph{$3$-connected} if it is simple and one needs to delete at least $3$ of its vertices to disconnect it. 
A toroidal map $M$ is called \emph{essentially $3$-connected}\footnote{In graph theory articles the term \emph{essentially $k$-connected} is used with a different meaning.} if $\Minf$ is $3$-connected. 
For instance the map in Figure~\ref{fig:maps}(a)
is not $3$-connected (since it has a double edge) but is essentially $3$-connected\footnote{A visual characterisation is that a toroidal map is essentially $3$-connected iff it admits a periodic planar straight-line drawing that is strictly convex (i.e., where all corners have angle smaller than $\pi$), which is satisfied in Figure~\ref{fig:maps}(a).}.  
A (possibly infinite) map is called \emph{irreducible} if its girth equals its minimum face-degree $d$,    
and there is no cycle of length $d$ apart from face-contours.

A \emph{contractible region} is a region homeomorphic to an open disk.
A \emph{contractible} closed walk of $M$
(resp. of $\Minf$) is defined as a 
 closed walk $W$ having a contractible region
on its right, which is called the interior of $W$.

A toroidal map $M$ is called \emph{essentially irreducible} if $\Minf$ is irreducible. Equivalently
in $M$, every closed walk that delimits a contractible region on its right side has length at least $d$ (with $d$ the minimum face-degree), 
and has length $d$ iff the enclosed region is a face.
 

In genus $0$ it is known that a map is 3-connected iff its angular map is irreducible. 
By an easy adaptation of the arguments to the periodic planar case
 (see also~\cite{RoVi90} for the higher genus case) one obtains:

\begin{claim}\label{claim:bijMQ}
For $M$ a toroidal map, $M$ is essentially $3$-connected iff its angular map 
is essentially irreducible.
\end{claim}

\section{Main results}
We let $\cH$ be the family of toroidal bipartite 6-quadrangular
maps that are essentially irreducible and such that, apart from the root-face contour, there is no 
other closed walk of length $6$ enclosing a contractible region containing the root-face.  
We let $\cQ$
be the family of toroidal bipartite quadrangulations that are essentially irreducible. 
 We denote by $\cT$ the family of  essentially $3$-connected toroidal maps.

A map is called \emph{unicellular}
if it has a unique face, and \emph{precubic} if all its vertices have degree in $\{1,3\}$. 
In a precubic map, the vertices of degree $1$ are called \emph{leaves} and those of degree $3$
are called \emph{nodes}; the edges incident to a leaf are called \emph{pending edges}, the other ones are called
\emph{plain edges}. We let $\cU$ be the family of precubic bipartite unicellular toroidal maps.
For $U\in\cU$ we say that $U$ is \emph{balanced} if any
(non-contractible) cycle $C$ of $U$ has the same number of incident edges on both sides, see
Figure~\ref{fig:closure}(a) for an example.   
Let $\cUbal$ be the set of balanced
elements in $\cU$.  

\subsection{Bijection $\psi$ from $\cUbal$ to $\cH$}\label{sec:bij_psi}
We describe here a bijection $\psi$ from $\cUbal$ to $\cH$, based on repeated closure operations, which can 
be seen as the toroidal counterpart of the bijection given in the planar case~\cite{FuPoScL}.

\begin{figure}[!ht]
\center
\includegraphics[scale=0.8]{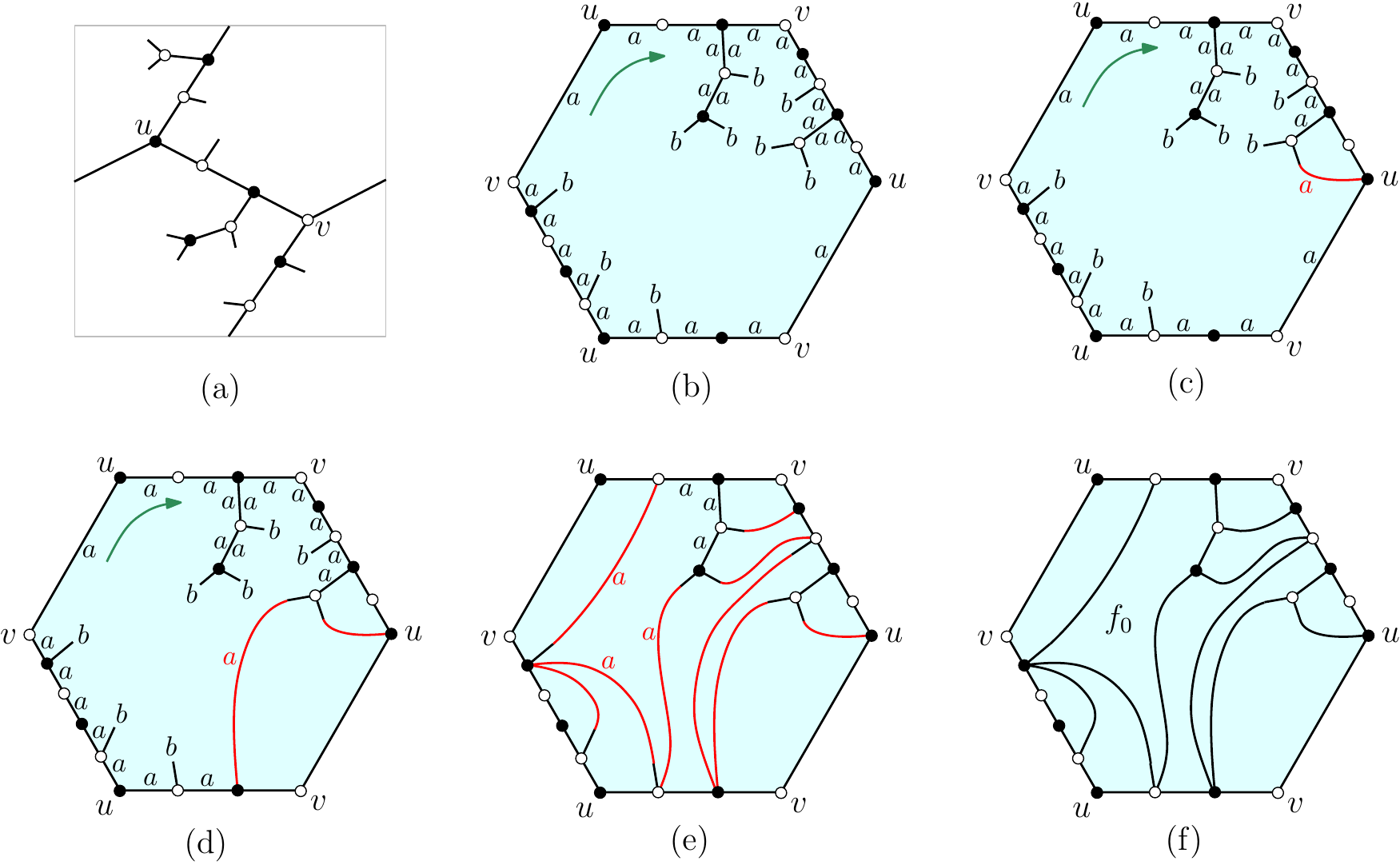}
\caption{(a) A unicellular map $U\in\cUbal$. (b) The same unicellular map in the representation where 
the unique face is unfolded into a hexagon (whose opposite sides are identified), where we write $b$ at each leaf
and write $a$ along every side of plain edge. (c) A local closure. (d) A second local closure. (e) The output after 
 greedily performing all local closures (the result does not depend on which order the local closures are done). (f)
The resulting map $\psi(U)=H\in\cH$.}
\label{fig:closure}
\end{figure}

Let $U\in\cUbal$, and let $m$ be the number of leaves of $U$. Then an application of the Euler formula together with the hand-shaking lemma implies
that $U$ has $m+3$ plain edges, hence $U$ has $2m+6$ sides of plain edges. Initially the \emph{root-face} 
of $U$ is defined as its unique face. A clockwise walk around the root-face of $U$ (i.e., with the 
root-face on our right) gives a cyclic word
on the alphabet $\{a,b\}$,  
writing $a$ when we pass along a side of plain edge, and writing
 $b$ when we pass along a leaf. Each time we see the pattern $baaa$ we can perform a so-called 
\emph{local closure} operation, which consists in merging the leaf at the end of the pending edge for $b$
to the end of the edge-side corresponding to the third $a$, forming a new quadrangular face, 
see Figure~\ref{fig:closure}(c). The resulting edge is now 
considered as a plain edge in the new figure, and accordingly the pattern $baaa$ is replaced by 
 the letter $a$ in the cyclic word coding the situation around the root-face. We can then repeat 
local closure operations (see Figure~\ref{fig:closure}(d)), 
each time decreasing by $1$ the numbers of $b$'s and by $2$ the number of $a$'s,
until there are no $b$'s. Note that the invariant $|a|=2|b|+6$ is maintained all along (which also ensures
that there is at least one local closure possible at each step). Hence, at the end, the map $H$ we obtain
is a bipartite toroidal 6-quadrangular map. We let $\psi$ be the mapping that associates $H$ to $U$  (see
Figures~\ref{fig:closure}(e) and~\ref{fig:closure}(f)).  

\begin{theorem}\label{theo:bij_psi}
The mapping $\psi$ is a bijection from $\cUbal$ to $\cH$. The number of black (resp. white) nodes in $\cUbal$ is mapped to the number of 
black (resp. white) vertices in $\cH$.
\end{theorem}

The inverse bijection $\phi$ will be described in Section~\ref{sec:inverse}. Then we will prove in Section~\ref{sec:proof_bij} that $\phi$ is a bijection from  $\cH$ to $\cUbal$, and prove in Section~\ref{sec:proof_inverse} 
that $\psi$ is its inverse. 

\subsection{Link between $\cH$ and  essentially $3$-connected toroidal maps}\label{sec:link}
We let $\cQ'$ be the family of maps from $\cQ$ with a marked edge, let $\cH'$ be the family of maps from $\cH$ with a marked corner incident to a white vertex in the hexagonal face,  and let $\cT'$
be the family of rooted maps from $\cT$ (i.e., with a marked corner). Note that each corner in a map corresponds to an edge in the angular map. 
Hence, Claim~\ref{claim:bijMQ} yields
\begin{equation}\label{eq:TQ}
\cT'\simeq\cQ',
\end{equation}
with the number of vertices (resp. faces) in the left-hand side corresponding to the number of white vertices
(resp. black vertices) in the right-hand side. 

We now explain how to decompose maps in $\cQ'$ into two parts, one planar and one in~$\cH'$. 

 For $Q\in\cQ'$, with $e$ its marked edge, 
an \emph{enclosing hexagon} of $Q$ is a closed walk of length $6$ bounding a contractible region that 
strictly contains $e$. An enclosing hexagon is called \emph{maximal} if its bounded  
region is not included in the bounded region of another enclosing hexagon.

\begin{lemma}
  \label{lem:maxdisjoint}
  Each $Q\in\cQ'$, has a unique maximal enclosing hexagon.
\end{lemma}

\begin{proof}
  Note that the union of the two faces of $Q$ incident to the marked edge $e_0$  
   forms an enclosing hexagon, so the set of enclosing hexagons is non-empty, and thus there
is at least one maximal enclosing hexagon. 

  We first reformulate the definition of a hexagon. Let us define a
  \emph{region} of $Q$ as given by $R=V'\cup E'\cup F'$ where
  $V',E',F'$ are subsets of the vertex-set, edge-set and face-set of
  $Q$, such that for $v\in V'$ all the edges incident to $v$ are in
  $E'$, and for $e\in E'$ the faces incident to $e$ are in
  $F'$. Note that the union (resp. intersection) of two regions is
  also a region.  We define a \emph{boundary-edge-side} of $R$ as an incidence
  face/edge of $Q$ such that the face is in $F'$ and the edge is not
  in $E'$. The \emph{boundary-length} of $R$, denoted by $\ell(R)$, is
  the number of boundary-edge-sides of $R$. Since $Q$ is bipartite,
  the value of $\ell(R)$ is even.
  A \emph{disk-region} is a
  region $R$ homeomorphic to an open disk. An enclosing hexagon thus corresponds
  to the (cyclic sequence of) boundary-edge-sides of a disk-region $R$
  such that $\ell(R)=6$ and $e_0\in E'$; and it is \emph{maximal} if there is no other
  disk-region ${R'}$ of boundary-length $6$ such that $R\subset {R'}$.
  It is easy to see that for any two regions $R_1,R_2$ we have
  $\ell(R_1)+\ell(R_2)=\ell(R_1\cup R_2)+\ell(R_1\cap R_2)$ (indeed any
  incidence face/edge of $Q$ has the same contribution to
  $\ell(R_1)+\ell(R_2)$ as to $\ell(R_1\cup R_2)+\ell(R_1\cap R_2)$).
  
For the sake of contradiction, suppose that there are two distinct maximal enclosing hexagons,
and let $R_1,R_2$ be the respective enclosed disk-regions. 
Since  $Q$ is essentially irreductible and both $R_1\cup R_2$ and
  $R_1\cap R_2$ contain $e_0$, we have
  $\ell(R_1\cup R_2) \geq 6$ and $\ell(R_1\cap R_2) \geq 6$. Moreover,
    $12=\ell(R_1)+\ell(R_2)=\ell(R_1\cup R_2)+\ell(R_1\cup R_2)$, so
  $\ell(R_1\cup R_2) = \ell(R_1\cap R_2) = 6$. Thus $R_1\cup R_2$ is
  an enclosing hexagon, which contradicts the
  maximality of $R_1,R_2$.
\end{proof}

Let $\cD$ be the family of planar bipartite irreducible 6-quadrangular maps with at least one edge not
incident to the hexagonal face,  
and let $\cD'$ be the set of maps from $\cD$ with a marked edge not incident to the hexagonal face. 

\begin{claim}\label{claim:decomp_HQ}
We have the following isomorphism:
\begin{equation}\label{eq:QH}
\cQ'\simeq \cH'\times \cD'.
\end{equation}
The number of black (resp. white) vertices in the left-hand side corresponds to three plus the total number of black (resp. white) vertices in the right-hand side
\end{claim}
\begin{proof}
For each $D\in\cD'$ we arbitrarily choose a vertex, denoted $v(D)$, among the $3$ white vertices incident to the hexagonal face\footnote{For instance the first white vertex incident to the hexagonal face encountered during a left-to-right depth-first traversal of the map $D$ starting from its root edge.} (considered as the outer face of $D$).  
To each pair $(H,D)\in\cH'\times\cD'$ we associate the bipartite quadrangulation $Q$ obtained by patching $D$ within the 
root-face of $H$, with $v(D)$ merged at the root-corner of $H$. Clearly the contour of the root-face of $H$ becomes
the maximal enclosing hexagon of $Q$. We also have to check that $Q$ is essentially irreducible, which amounts to check that there is no closed walk of length at most $4$ enclosing a contractible region that covers both faces in $H$ and in $D$. An easy case-analysis ensures that it is not possible without creating a closed walk of length at most $6$ in $H$ that does not bound the hexagonal face but encloses a contractible region containing the hexagonal face, a contradiction. 

Conversely, from $Q\in\cQ'$, we let $C_0$
be the maximal enclosing hexagon of $Q$. Let $D$ be the map in $\cD'$ formed by $C_0$ (unfolded into a simple 6-cycle) and the faces within the contractible region enclosed by $C_0$. And let $H\in\cH'$ be the map obtained by emptying
the area enclosed by $C_0$; the emptied area is now a hexagonal face where we mark the white corner where $v(D)$ was. 

The mappings from $(H,D)$ to $Q$ and from $Q$ to $(H,D)$ being clearly inverse of each other, we have a bijection between 
$\cQ'$ and $\cH'\times\cD'$. 
\end{proof}

\subsection{Counting rooted essentially $3$-connected toroidal maps}\label{sec:counting}

\subsubsection{Expression for the bivariate generating function}

In this section we show that the bijection of Theorem~\ref{theo:bij_psi}, combined with the 
decomposition of Section~\ref{sec:link} and the planar case bijection~\cite{FuPoScL}, allows us to derive 
an explicit  algebraic expression for the (bivariate) generating function of rooted  essentially
$3$-connected toroidal maps:

\begin{theorem}
\label{thm:counting}
Let $T\equiv T(\zb,\zw)$ be the generating function of $\cT'$, with $\zb$ and $\zw$ conjugate to the numbers of 
 faces and vertices respectively. Then 
\begin{equation}\label{eq:countij}
T(\zb,\zw) = \frac{\rb\rw(\rb^2 + \rw^2 + \rb\rw+2\rb +2\rw + 1)}{(\rb+\rw+1)(1+\rb+\rw-3\rb\rw)^2}=\frac{p}{(s-3p)^2}(s-p/s),
\end{equation}
where $\rb \equiv \rb(\zb,\zw)$ and $\rw \equiv \rw(\zb,\zw)$ are the algebraic series specified by the system 
$\{\rb=\zb(1+\rw)^2,\ \rw=\zw(1+\rb)^2\}$, and where $s=1+\rb+\rw$ and $p=\rb\rw$.
\end{theorem}

The initial terms of the series are 
\begin{align*}
T=&\zb\zw+(\zb^2\zw+\zb\zw^2)+11\zb^2\zw^2+(20\zb^3\zw^2+20\zb^2\zw^3)+(10\zb^4\zw^2+146\zb^3\zw^3+10\zb^2\zw^4)\\
&+(329\zb^4\zw^3+329\zb^3\zw^4)+(300\zb^5\zw^3+2047\zb^4\zw^4+300\zb^3\zw^5)+\cdots,
\end{align*}
see Figure~\ref{fig:first_examples} below for an illustration (maps in $\cT'$ with at most $4$ edges).

\begin{figure}[!h]
\center
\includegraphics[scale=0.6]{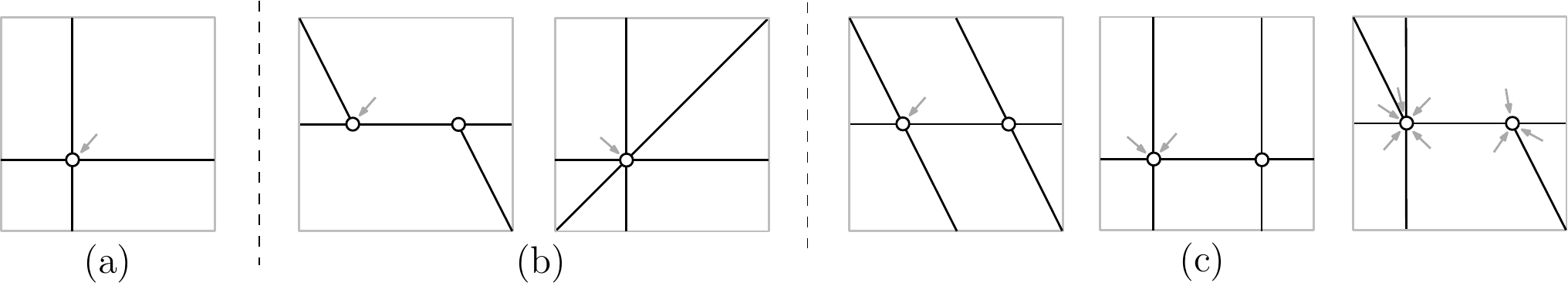}
\caption{(a) The unique map in $\cT'$ with $2$ edges. (b) The two maps in $\cT'$ with $3$ edges (the left one 
has one face and two vertices, the right one has two faces and one vertex). (c) The $1+2+8=11$ maps in $\cT'$ with $4$ edges (all have two vertices and two faces).}
\label{fig:first_examples}
\end{figure}

A rational expression in terms of $\rb,\rw$ actually exists in any genus.    
Precisely, in genus~$g$ we call a map  \emph{essentially $3$-connected} 
 if it is $3$-connected in the periodic representation 
(in the Poincar\'e disk when $g>1$), 
and let $T^{(g)}\equiv T^{(g)}(\zb,\zw)$ be the associated bivariate generating function with respect to the numbers of faces and vertices. In genus~$g$ we call a quadrangulation \emph{essentially irreducible} if 
it has no closed walk of length $\leq 4$ enclosing a contractible region that is not a face. 
The bijection of Claim~\ref{claim:bijMQ} holds in any genus~\cite{RoVi90},  
so that $T^{(g)}$ is also the series of edge-rooted bipartite essentially irreducible quadrangulations of genus $g$, 
counted by black vertices and white vertices. It can be shown that $T^{(g)}(\zb,\zw)$ has an explicit rational expression
 in terms of $\rb$ and $\rw$. 
This can be done by a substitution approach (see e.g.~\cite{chapuy2011asymptotic}) 
that relates $T^{(g)}$ to the bivariate series of edge-rooted bipartite quadrangulations of genus $g$, for which an explicit  algebraic expression is known~\cite{bender1993asymptotic,arques1999enumeration}.  Our 
derivation (detailed in Section~\ref{sec:deriv}) is the first bijective one in genus $1$ 
(a bijective derivation in genus $0$ is given in~\cite{FuPoScL}). 

\subsubsection{Univariate specializations}
We mention here some univariate specializations of Theorem~\ref{thm:counting}. First, note that $T_e(z):=T(z,z)$
is the generating function of $\cT'$ by edges, the expression in Theorem~\ref{thm:counting} becomes
\begin{equation}
T_e(z) = \frac{r^2(1+r)}{(1+2r)(1-r)^2(1+3r)},\ \ \ \mathrm{with}\ r\equiv r(z)\ \mathrm{given\ by\ }r=z(1+r)^2.
\end{equation}The initial terms are 
$T_e(z)=z^2+2z^3+11z^4+40z^5+166z^6+658z^7+2647z^8+10592z^9+\cdots$. This is sequence A308524 in the OEIS. 

On the other hand, $T_v(z):=T(1,z)$ is the generating function of $\cT'$ by vertices, and
 the expression in Theorem~\ref{thm:counting} becomes
\begin{equation}
T_v(z) = \frac{(r+1)(r^2+3r+4)r}{(3r^2+2r-2)^2(r+2)},\ \ \ \mathrm{with}\ r\equiv r(z)\ \mathrm{given\ by\ }r=z(2+2r+r^2)^2.
\end{equation}
 The initial terms are 
$T_v(z)=2z+42z^2+892z^3+18888z^4+399280z^5+8431776z^6+\cdots$. This is sequence A308526 in the OEIS.

We now consider the specialization to triangulations. Similarly as in genus $0$, a toroidal triangulation 
is essentially 3-connected iff it is essentially simple. 
Note that in a toroidal essentially 3-connected map, all faces have degree at least $3$, hence the numbers $i,j$ of vertices and faces satisfy $2i\geq j$, with equality iff the map is a triangulation.  
Hence if we let $T_t(z)$ be the generating function of rooted  essentially simple toroidal triangulations counted by vertices, then we have 
\[T_t(z)=\lim_{a\to 0}T(1/a,za^2).\]
 We have
$\{\rb=a^{-1}(1+\rw)^2,\ \rw=za^2(1+\rb)^2\}$, hence $\rw=z(a+(1+\rw)^2)^2$. If we let $r:=\lim_{a\to 0}\rw$,
then we have $r=z(1+r)^4$. Since $\rb$ dominates $\rw$ as $a\to 0$ ($\rb$ being of order $1/a$), we find, as $a\to 0$, 
\[
p\sim \rb r,\ \ s-p/s\sim \rb,\ \  (s-3p)^2\sim \rb^2(1-3r)^2. 
\]
Hence, from the expression in Theorem~\ref{thm:counting} we obtain
\begin{equation}\label{eq:serT}
T_t=\frac{r}{(1-3r)^2},\ \ \ \mathrm{with}\ r\equiv r(z)\ \mathrm{given\ by\ }r=z(1+r)^4.
\end{equation}
The initial terms are $T_t=z\!+\!10z^2\!+\!97z^3\!+\!932z^4\!+\!8916z^5\!+\!85090z^6\!+\!810846z^7+\cdots$. This is sequence A308523 in the OEIS.  
This expression of $T_t$ has also recently been derived bijectively in~\cite[Proposition 26]{FL18}, and actually the bijection  for triangulations given there can be seen as a specialization of the bijection we develop here, as we will see in Section~\ref{sec:bij_tr}.

\subsubsection{Proof of Theorem~\ref{thm:counting}}\label{sec:deriv}
The combinatorial proof of Theorem~\ref{thm:counting} that we present in this section can be seen as a bivariate adaptation of the proof of~\eqref{eq:serT} given in~\cite[Sect.5.2]{FL18} (itself an adaptation of the calculations in~\cite{CMS09} for
bipartite quadrangulations of genus $g$).  
A first remark is the combinatorial interpretation of the two series $\rb,\rw$. 
A \emph{bipartite binary tree} is a planar bipartite precubic map. A black-rooted (resp. white-rooted) binary tree
is a bipartite binary tree with a marked pending edge incident to a black (resp. white) node. 
Then clearly $\rb$ and $\rw$ are the generating functions of black-rooted and white-rooted binary trees,
counted with respect to the numbers of black nodes and white nodes. We also let $\Rb=1+\rb$ and $\Rw=1+\rw$. 

Let $D\equiv D(\zb,\zw)$ be the generating function of $\cD'$ with $\zb$ (resp. $\zw$) conjugate to the number of 
black (resp. white) vertices that are not incident to the hexagonal face. Then it follows from~\cite[Sect.5.1]{FuPoScL} 
that $\cD'$ is in bijection with bipartite binary trees with a marked edge (pending or plain), giving
\[
D=\Rb\Rw.
\]
We now let $Q\equiv Q(\zb,\zw)$ (resp. $H\equiv H(\zb,\zw)$) be the generating function of $\cQ'$
(resp. of $\cH'$) with $\zb,\zw$ conjugate to the numbers of black and white vertices, respectively. 
Note that $T=Q$ according to~\eqref{eq:TQ}. Moreover, it follows from~\eqref{eq:QH} that $Q=H\ \!\!D$, hence we have
\begin{equation}\label{eq:T}
T=\Rb\Rw H.
\end{equation}
It thus remains to compute $H$. For a toroidal unicellular map $U$,
the \emph{core} $C$ of $U$ is obtained from $U$
by successively deleting leaves, until there is 
no leaf (so $C$ has all its vertices of degree at
least $2$; the deleted edges form trees attached
at vertices of $C$). In $C$ we call \emph{maximal chain}
a path $P$ whose extremities have degree larger than $2$
and all non-extremal vertices of $P$ have degree $2$. 
Then the \emph{kernel} $K$ of $U$ is obtained
from $C$ by replacing every maximal chain by an edge.
The kernel of a toroidal unicellular
map is either made of one vertex with two loops (double loop) 
or is made of $2$ vertices and $3$ edges joining them
(triple edge). Note that in the first case, there is a vertex of
degree at least four. So this case never occurs for elements of
$\mathcal U$. Thus elements of $\mathcal U$ have six half-edge in
their associated kernel. A map in $\mathcal{U}$ is called \emph{kernel-rooted} if 
one of the six half-edges is marked. 

Let $N(\zb,\zw)$ be the generating function of kernel-rooted maps from $\cUbal$ with $\zb,\zw$ conjugate 
to the numbers of black and white vertices, respectively. The bijection of Theorem~\ref{theo:bij_psi} (on unrooted objects) and a classical 
double-counting argument (considering the number of potential rootings, choice among $6$ half-edges for kernel-rooted maps, and among $3$ white corners 
for $6$-quadrangular maps) ensure that $6H=3N$, so that 
\begin{equation}\label{eq:H}
H=N/2.
\end{equation}
Hence it remains to express $N$ in terms of $\rb,\rw$.  
Note that the generating function $N$ splits as
\begin{equation}\label{eq:N}
N=\Nbb+\Nwb+\Nbw+\Nww=\Nbb+2\Nbw+\Nww,
\end{equation}
depending on the colors of the two vertices $v_1,v_2$ of the 
kernel (with $v_1$ the one incident to the marked half-edge), and 
where the second equality follows from $\Nbw(z)=\Nwb(z)$, 
since $v_1$ and $v_2$ play symmetric roles.

A \emph{skeleton} is a toroidal unicellular precubic map such that every node belongs to the core. 
Observe that a map $U\in\cU$ is balanced if and only if its skeleton is balanced. We let $\Sbb,\Sbw,\Sww$
be the generating functions gathering the respective contributions from $\Nbb/(\zb^2),\Nbw/(\zb\zw),\Nww/(\zw^2)$ that are skeletons. 
We clearly have

\begin{equation}\label{eq:NrelS}
\Nww = \zw^2 \Sww(\zb\Rw,\zw\Rb),\ \Nbb = \zb^2 \Sbb(\zb\Rw,\zw\Rb),\ \Nbw = \zb\zw \Sbw(\zb\Rw,\zw\Rb).
\end{equation}

A \emph{bi-rooted caterpillar} is a bipartite binary tree with two marked leaves $v_1,v_2$, called \emph{primary root}  and \emph{secondary root}, such that every node is on the path $\pi$ from $v_1$ to $v_2$. The \emph{$\gamma$-score} of a bi-rooted caterpillar is the number of non-root leaves on the right-side of $\pi$ minus the number of non-root leaves
on the left-side of $\pi$. 

Clearly a skeleton $S$ (with a marked half-edge in the kernel)  
 decomposes into an ordered triple of  bi-rooted caterpillars,   
and $S$ is balanced if and only if the $3$ bi-rooted caterpillars have the 
same $\gamma$-score.  
Hence, if for $i\in\mathbb{Z}$ we 
let $\Cbbi(\tb,\tw)$, $\Cbwi(\tb,\tw)$, $\Cwwi(\tb,\tw)$ be the generating functions 
of bi-rooted caterpillars of $\gamma$-score $i$ 
where $v_1,v_2$ are black/black (resp.
black/white, white/white), and with $\tb$ (resp. $\tw$) conjugate to the number of black (resp. white) nodes, then we find

$$\Sbb(\tb,\tw)\!=\!\sum_{i\in\mathbb{Z}}\Cbbi(\tb,\tw)^3,\ \ \Sbw(\tb,\tw)\!=\!\sum_{i\in\mathbb{Z}}\Cbwi(\tb,\tw)^3,\ \ \Sww(\tb,\tw)\!=\!\sum_{i\in\mathbb{Z}}\Cwwi(\tb,\tw)^3.$$


For $i\in\zZ$, let $p_{n,i}$ be the number
 of walks of length $n$ with steps in $\{-1,1\}$,
starting at $0$ and ending at $i$ (note that 
$p_{n,i}=0$ if $i\neq n\mathrm{\ mod\ }2$). 
We also define the generating
function of walks ending at $i$ as 
$$
\Pi(t)=\sum_{n\geq 0}p_{n,i}t^{\lfloor n/2\rfloor}.
$$ 
We clearly have for $i\in\zZ$, 

$\begin{array}{ll}
\Cwwi(\tb,\tw)=0\ \mathrm{for}\ i\ \mathrm{even},& 
\Cwwi(\tb,\tw)=\tb\cdot\Pi(\tb\tw)\ \ \mathrm{for}\ i\ \mathrm{odd}.\\
\Cbbi(\tb,\tw)=0\ \mathrm{for}\ i\ \mathrm{even},&
\Cbbi(\tb,\tw)=\tw\cdot\Pi(\tb\tw)\ \ \mathrm{for}\ i\ \mathrm{odd}.\\
\Cbwi(\tb,\tw)=0\ \mathrm{for}\ i\ \mathrm{odd},&
\Cbwi(\tb,\tw)=\Pi(\tb\tw)\ \ \mathrm{for}\ i\ \mathrm{even}.
\end{array}$

We have $\Pi(t)=P^{(-i)}(t)$ for $i<0$, 
and for $i>0$ we classically have (see~\cite{PoSc04} for more details on the decomposition at the last visits to $0,1,\dots,i-1$) 
$$
\Pi(t)=B\cdot(1+U)^i\cdot t^{\lfloor i/2\rfloor},
$$
where $U=t(1+U)^2$ is the series of non-empty Dyck paths, and $B=\frac1{1-2t\cdot(1+U)}=\frac{1+U}{1-U}$
is the series of bridges ($B$ is $P^{(0)}$), with $t$ conjugate to the half-length.

Hence we have
\[
\Sww=\tb^3
\sum_{\substack{i\in\zZ\\i\ \mathrm{odd}}}\Pi(t)^3\Big|_{t=\tb\tw}=2\tb^3\frac{B^3\cdot(1+U)^3}{1-t^3(1+U)^6}\Big|_{t=\tb\tw}=2\tb^3\frac{B^3\cdot(1+U)^3}{1-U^3}\Big|_{t=\tb\tw}
\]
and replacing $B$ by $(1+U)/(1-U)$ we obtain
\[
\Sww=\frac{2\tb^3(U + 2 + U^{-1})^3}{(U+1+U^{-1})(U-2+U^{-1})^2}.
\]

Similarly as in~\cite[Sect.4.5]{CMS09}, we observe that $\Sww/\tb^3$ is rational
in $U+U^{-1}$,  so it is also rational in $\tb\tw$ since
$U+U^{-1}=1/t-2$ and $t=\tb\tw.$ Overall we obtain

$$\Sww = \frac{2\tb^3}{(1-\tb\tw)(1-4\tb\tw)^2}.$$



Using~\eqref{eq:NrelS} we deduce then

$$\Nww = \frac{2\zw^2\zb^3\Rw^3}{(1-\Rb\Rw\zb\zw)(1-4\Rb\Rw\zb\zw)^2}=\frac{2\rb^3\rw^2}{(1+\rb)(1+\rb+\rw)(1+\rb+\rw-3\rb\rw)^2},$$
where we have used $\Rb=\rb+1$, $\Rw=\rw+1$, $\zb=\rb/(1+\rw)^2$, $\zw=\rw/(1+\rb)^2$.



Very similarly we get rational expressions for $\Nbb$ and $\Nbw$ in terms of $\rb,\rw$.
Adding them up we obtain
\[
N=\frac{2\rb\rw(1+2\rb+2\rw+\rb\rw+\rb^2+\rw^2)}{(1+\rb)(1+\rw)(1+\rb+\rw)(1+\rb+\rw-3\rb\rw)^2},
\]
which together with~\eqref{eq:T} and~\eqref{eq:H} completes the proof of Theorem~\ref{thm:counting}.






\section{The inverse bijection $\phi$ from $\cH$ to $\cUbal$}\label{sec:inverse}

In this section we present the inverse bijection $\phi$ of $\psi$. The different steps to define this bijection are illustrated in~Figure~\ref{fig:inverse}.

In the first step, we add a black dummy vertex in the hexagonal face of $H$ and connect it to its 3 white vertices (see Figure~\ref{fig:inverse}~(a)). Then, considering the angular map of the resulting map, we compute the unique minimal balanced Schnyder orientation (see Figure~\ref{fig:inverse}~(b) and Section~\ref{ssec:orientations}). Thanks to some local rules (see Figure~\ref{fig:local_rule_biorient}) we compute a particular biorientation of $H$ that is called \emph{canonical} (see Figure~\ref{fig:inverse}~(c) and Section~\ref{ssec:biorientations}). The last step of the bijection consists in removing the ingoing half-edges of the biorientation to obtain an element of $\cUbal$.  Finally in Section~\ref{sec:bij_tr} we show that $\phi$ can be specialized so as to recover the bijection for toroidal triangulations given in~\cite[Theo.19]{FL18}.

\subsection{Orientations}\label{ssec:orientations}
Let $M$ be a toroidal map endowed with an orientation  $X$. 
For $C$ a non-contractible cycle of $M$ given with a
traversal direction of $C$, we denote by $\gamma_R(C)$
(resp. $\gamma_{L}(C)$) the total number of edges going out of a
vertex on $C$ on the right (resp. left) side of $C$, and we define the
\emph{$\gamma$-score} of $C$ as $\gamma(C)=\gamma_R(C)-\gamma_L(C)$. 

For $G=(V,E)$ a graph and $\alpha:V\to \mb{N}$, an \emph{$\alpha$-orientation} of $G$ (see~\cite{Fe03}) is an orientation of $G$ such that every vertex $v\in V$ has outdegree $\alpha(v)$.  
For $M$ a toroidal map, two $\alpha$-orientations $X,X'$ of $M$ are
 called \emph{$\gamma$-equivalent} if every non-contractible cycle of
 $M$ has the same $\gamma$-score in $X$ as in $X'$. 

Assume $M$ is a face-rooted toroidal map, with $f_0$ its root face.  An orientation
of $M$ is called \emph{non-minimal} if there exists a non-empty set $S$ of faces
such that $f_0\notin S$ and every edge on the boundary of $S$ has a
face in $S$ on its right (and a face not in $S$ on its left).  It is
called \emph{minimal} otherwise. 

The following result is Corollary~3 in~\cite{FL18} (as explained in~\cite[Sect.2.2]{FL18}, it follows from a closely related result  proved in~\cite{Pro93} and stated as Theorem~2 in~\cite{FL18}):

\begin{theorem}[\cite{FL18}]
  \label{theo:gamma}
Let $M$ be a face-rooted toroidal map that admits an $\alpha$-orientation $X$.  
Then $M$ has a unique $\alpha$-orientation
 $X_0$ that is minimal and $\gamma$-equivalent to $X$.

 Moreover, for two $\alpha$-orientations $X,X'$ of $M$ to be $\gamma$-equivalent, it 
is enough that $X$ and $X'$ have the same $\gamma$-score on two non-contractible
non-homotopic\footnote{Two closed curves on a surface are called \emph{homotopic} if one can be continuously deformed into the
	other.} cycles of $M$.
\end{theorem}

For $M$ an essentially $3$-connected toroidal map, with $\hM$ its derived map, 
a \emph{Schnyder orientation} of $\hM$ is an orientation of $\hM$ such that every 
primal or dual vertex has outdegree $3$, and every edge-vertex has outdegree $1$. 
A Schnyder orientation is called \emph{balanced} if for every non-contractible cycle
$C$ not passing by dual vertices, one has $\gamma(C)=0$. 

The following result is shown in~\cite[Theorem 17, Lemma
8 and Proposition 11]{LevHDR} (the algorithm to compute the orientation is based on repeated 
contraction operations done according to a careful case analysis):

 \begin{theorem}[\cite{LevHDR}]
\label{th:existence}
   Let $M$ be an   essentially 3-connected  toroidal map. Then $\hM$ 
 admits a balanced Schnyder orientation. 
\end{theorem}

Then a direct consequence of
Theorem~\ref{theo:gamma} is the following:

 \begin{theorem}
   \label{th:unicity}
 Let $M$ be a face-rooted essentially 3-connected toroidal map. Then $\hM$ 
 admits a unique minimal balanced Schnyder orientation.
\end{theorem}

An example is shown in Figure~\ref{fig:inverse}(b).


\subsection{Biorientations}\label{ssec:biorientations}
A \emph{biorientation} of a map is an orientation of its half-edges (each half-edge is either directed outward or inward the incident vertex) such that for each edge, at least one of its half-edges is outgoing. An edge is \emph{bidirected} if both of its half-edges are outgoing. An edge with an ingoing half-edge is \emph{simply directed}.
The \emph{outdegree} of a vertex $v$ is the number of outgoing half-edges incident to $v$. 
The \emph{ccw-degree} of a face $f$ is the number of simply directed edges that have $f$ on their
left. 
A \emph{3-biorientation} is a biorientation where every vertex has outdegree $3$. A 3-biorientation of a toroidal 
quadrangulation, or of a toroidal 6-quadrangular map, is called \emph{S-quad} if every quadrangular face
has ccw-degree $1$. 
For a 6-quadrangular toroidal map, an easy argument based on the Euler
relation ensures that the ccw-degree of the root-face has to be $0$ in any S-quad 3-biorientation.

Let $M\in \cT$ and let $Q\in\cQ$ be the angular map of $M$ (we recall that 
$\hM$ is the angular map of $Q$). For $X$ an S-quad 3-biorientation of $Q$,
we let $Y=\sigma(X)$ be the orientation of $\hM$ obtained by applying
the local rules shown in Figure~\ref{fig:transfer} (where $Q$ is
represented with red edges and $\hM$ with black edges). 

\begin{figure}[!h]
\center
\includegraphics[scale=0.98]{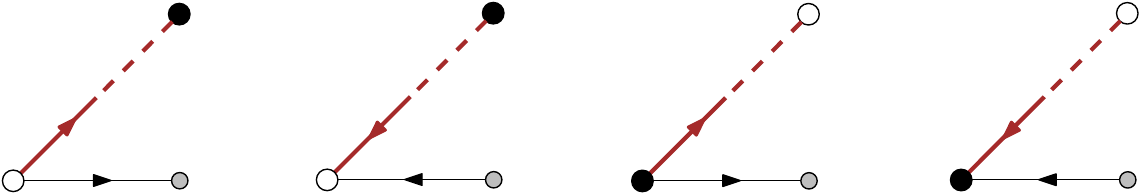}
\caption{The local rules in the mapping $\sigma$. }
\label{fig:transfer}
\end{figure}

 The fact that all vertices have outdegree $3$ in $X$ implies that the primal and dual vertices have
outdegree $3$ in $Y$. And the fact that every face of $X$ has ccw-degree $1$ implies that every 
edge-vertex in $Y$ has outdegree $1$. Finally the fact that the edges of $X$ have at least one outgoing 
half-edge implies that no face in $Y$ is clockwise. Conversely, starting from a Schnyder orientation of $\hM$ 
with no clockwise face and applying the local rules one obtains an S-quad 3-biorientation of $Q$. To summarize:

\begin{claim}
The mapping $\sigma$ is a bijection between the S-quad 3-biorientations of $Q$ and the Schnyder orientations
of $\hM$ with no clockwise face. 
\end{claim}
An example is shown in Figure~\ref{fig:inverse}(c). 

An S-quad 3-biorientation $X$ of $Q$ is called \emph{balanced} if the corresponding Schnyder
orientation $\sigma(X)$ is balanced. 

\begin{figure}
\center
\includegraphics[scale=0.77]{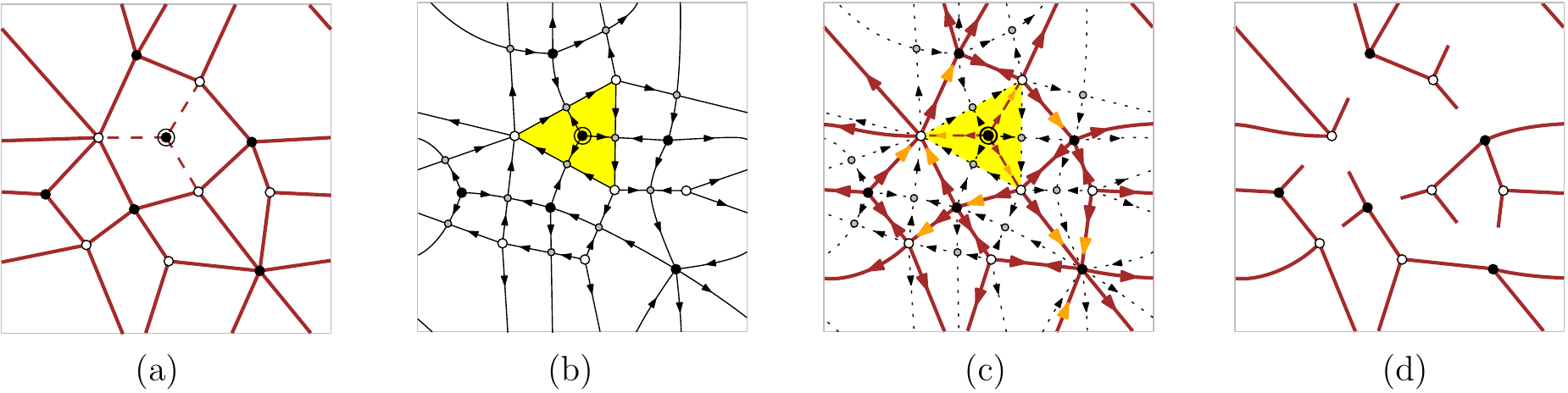}
\caption{(a) A toroidal map $H\in\cH$; adding the surrounded black vertex and the 3 dashed edges one obtains 
a bipartite quadrangulation $Q\in\cQ$. (b) The derived map $\hM$ (with $M$ the map whose angular map is $Q$), endowed with its minimal
balanced Schnyder orientation (minimal with respect to any of the 3 faces incident to the 
surrounded black vertex). (c) The canonical 3-biorientation of $Q$ (and then of $H$, deleting
the surrounded vertex and $3$ incident edges), obtained by applying the local rules of Figure~\ref{fig:transfer}; ingoing
half-edges are colored orange. (d) The unicellular map $\phi(H)=U\in\cUbal$ is obtained from $H$ 
by deleting the ingoing half-edges.}
\label{fig:inverse}
\end{figure} 

\subsection{Description of the bijection $\phi$}\label{sec:desc_phi}
Let $H\in\cH$, and let $Q$ be the bipartite quadrangulation 
obtained by adding a black vertex $v_0$ of degree $3$ inside the root-face
$f_0$ of $H$, and connecting $v_0$ to the $3$ corners $c_1,c_2,c_3$ at
white vertices around $f_0$, see Figure~\ref{fig:inverse}(a). 
If we let $B$ be the map in $\cD$ with a unique internal vertex that is black, then $Q$ is obtained
by patching $B$ within the root-face of $H$, hence $Q$ is in $\cQ$ according to Claim~\ref{claim:decomp_HQ}. 

Let $M$ be the map whose angular map is $Q$. Since $\hM$ is the angular map of $Q$, each face of $\hM$
corresponds to an edge of $Q$. We choose (arbitrarily) the root-face $f_0$ among the $3$ faces of $\hM$ 
incident to $v_0$. We let $Y$ be the minimal balanced
Schnyder orientation of $\hM$, see Figure~\ref{fig:inverse}(b).
 Since $v_0$ is a source in $Y$, the root-face contour is not a clockwise cycle. In addition 
the contours of the other faces are neither clockwise cycles by minimality. Hence $Y$ has no clockwise face.  
The fact that $v_0$ is a source also implies that $Y$ is minimal for any of the $3$ root-face choices, hence $Y$ does not depend on which of the $3$ faces incident to $v_0$ is chosen.
 Let $X'=\sigma^{-1}(Y)$ be the associated S-quad 3-biorientation of $Q$ (obtained using the rules of Figure~\ref{fig:transfer}).  
We let $X$ be the biorientation of $H$, called the \emph{canonical biorientation} of $H$, which is obtained from $X'$ by deleting $v_0$ and its 3 incident edges. We will see in Section~\ref{sec:suffi} (Lemma~\ref{lem:inO6}) that in $X'$, the 3 edges at $v_0$ are simply directed (out of $v_0$), hence $X$ is an S-quad 3-biorientation. 


We can now describe the mapping $\phi$. For $H$ endowed with its canonical biorientation $X$, we simply
let $\phi(H)$ be the map obtained by deleting all the ingoing half-edges (thus any simply directed edge is turned
into a pending edge), see Figure~\ref{fig:inverse}(c)-(d). 

\begin{theorem}\label{theo:bij_phi}
The mapping $\phi$ is a bijection from $\cH$ to $\cUbal$. The number of black (resp. white) vertices in $\cH$ is mapped
to the number of black (resp. white) nodes in $\cUbal$.
\end{theorem}

We prove this result in Section~\ref{sec:proof_bij}. 

\subsection{Specialization to triangulations}\label{sec:bij_tr}
As already mentioned, a toroidal triangulation $M$ is essentially 3-connected iff it is 
 \emph{essentially simple}, i.e., $\Minf$ is simple. Let $\cTri$
be the family of face-rooted essentially simple toroidal triangulations such that, apart from 
the root-face contour, there is no other closed walk of length $3$ enclosing a contractible 
region that contains the root-face. 
Let $\cUbaltri$ (resp. $\cHtri$) be the subfamily of $\cUbal$ (resp. $\cH$)
 where the respective numbers $i,j$ of white vertices
and black vertices satisfy $j=2i-1$. For $\cHtri$ this amounts to having all black vertices of degree $3$, 
and for $\cUbaltri$ this amounts to having no pending edge incident to a white node. 
The bijection of Theorem~\ref{theo:bij_phi} specializes into a bijection between  $\cHtri$ and $\cUbaltri$. 
For $M\in\cTri$, we let $Q$ be its angular map,  with  
$v_0$ the black vertex corresponding to the root-face of $M$, and let $H=\iota(M)$ be the map obtained from $Q$ 
by deleting $v_0$ and its 3 incident edges. Then it can be checked that $H\in\cHtri$ and that $\iota$ gives
a bijection between $\cTri$ and $\cHtri$. Thus the specialization of $\phi$ yields a bijection
between $\cTri$ and $\cUbaltri$. We actually recover here the bijection between $\cTri$ and $\cUbaltri$
recently given in~\cite{FL18}.




\section{Proof of Theorem~\ref{theo:bij_phi}}\label{sec:proof_bij}
Our approach to show Theorem~\ref{theo:bij_phi} is as follows. First, in Section~\ref{sec:biorientations} we recall from~\cite{FL18} a bijection (working for any fixed $d\geq 1$) between a certain family $\cO_d$ of bioriented toroidal maps with root-face degree $d$, and a certain family of decorated unicellular toroidal maps, which are called (toroidal) bimobiles. Then we prove in Section~\ref{sec:spec_Squad} that, for $d=6$, this bijection specializes into a bijection between the family ---denoted by $\cX$--- of 6-quadrangular bipartite toroidal maps endowed with an S-quad 3-biorientation in $\cO_6$, and a family of toroidal bimobiles that naturally identifies to the family $\cU$ (bipartite precubic unicellular toroidal maps).  

We then establish a further specialization: the subfamily $\cXbal$ of $\cX$ that is mapped to $\cUbal$ identifies to the family $\cH$. This relies on the two following properties: for a biorientation in $\cXbal$ the underlying map is in $\cH$ (Lemma~\ref{lem:necess}); and any map in $\cH$ admits a unique biorientation in $\cXbal$ (Lemma~\ref{lem:suffi} proved in Section~\ref{sec:proof_lem_suffi}, where we also show that the unique biorientation in $\cXbal$ is actually the canonical 3-biorientation considered in the description of $\phi$). 
This further specialization thus yields a bijection between $\cH$ and $\cUbal$. Moreover, as explained in Section~\ref{sec:spec_Squad} (and illustrated in Figure~\ref{fig:two_versions_phi}), the fact that the unique biorientation in $\cXbal$ is the canonical 3-biorientation easily implies that this bijection coincides with the mapping $\phi$ described in Section~\ref{sec:desc_phi}.

\subsection{Bijection between right biorientations and bimobiles}
\label{sec:biorientations}

Let $M$ be a map endowed with a biorientation such that every vertex has at least one outgoing
half-edge.  For an outgoing half-edge $h$ of $M$, we define the
\emph{rightmost walk} from $h$ as the (necessarily unique and
eventually looping) sequence $h_0,h_1,h_2,\ldots$ of half-edges starting from $h$, 
 at each step taking the opposite half-edge and then the rightmost
 outgoing half-edge at the current vertex; in other words, 
 for each $i\geq 0$, $h_{2i+1}$ is opposite to $h_{2i}$, and $h_{2i+2}$ is the next outgoing
 half-edge after $h_{2i+1}$ in counterclockwise order around their incident vertex.

If $M$ is face-rooted, a biorientation of $M$ is called a \emph{right biorientation} 
if  every vertex has at least one outgoing half-edge, and for every outgoing half-edge $h$, the rightmost walk starting
  from $h$ eventually loops on the contour of the root-face $f_0$ with $f_0$ on
  its right side. 
For $d\geq 1$,
 $\mathcal O_d$ denotes the family of right biorientations
of toroidal face-rooted maps whose root-face has degree~$d$.

We call \emph{(toroidal) bimobile}  a toroidal unicellular
map with two kinds of vertices, round or square, with no square-square edge, and such that each
corner at a square vertex might carry additional dangling half-edges
called \emph{buds}.  The \emph{excess} of a bimobile is the number of
round-square edges plus twice the number of round-round edges, minus the
number of buds.  We let $\cB_d$ be the family of bimobiles of excess $d$.

For $O\in\mathcal O_d$ (whose vertices are considered round) we denote by $\Phi_+(O)$ the embedded graph
obtained by inserting a square vertex in each face of $O$, then applying
the local rules of Figure~\ref{fig:local_rule_biorient} to every edge
of $O$ (thereby creating some edges and buds), and finally erasing the
isolated square vertex in the root-face of $O$ (since the orientation is
right, this square vertex is incident to $d$ buds and no edge), see Figure~\ref{fig:ex_Phi+} for an example. 

\begin{figure}[!h]
\center
\includegraphics[scale=0.98]{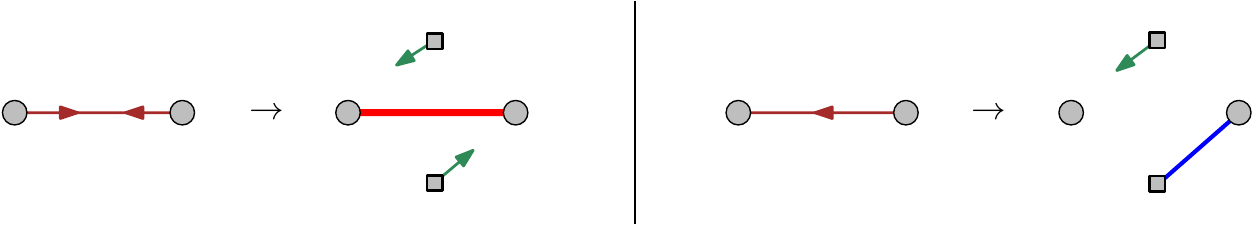}
\caption{The local rules applied to each edge by $\Phi_+$.}
\label{fig:local_rule_biorient}
\end{figure}

\begin{figure}[!h]
\center
\includegraphics[scale=0.98]{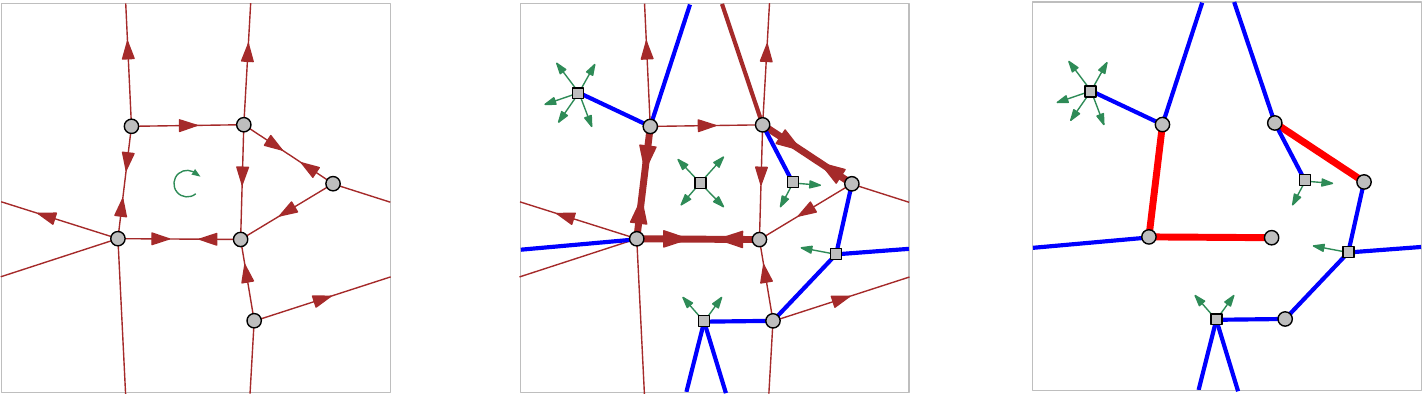}
\caption{Left: a face-rooted toroidal bioriented map $X$ in $\cO_4$. Right: the corresponding
 bimobile $\Phi_+(X)$, obtained by applying the local rules of Figure~\ref{fig:local_rule_biorient}.}
\label{fig:ex_Phi+}
\end{figure}

The following result has been obtained in~\cite{FL18}, derived from the bijection for covered maps given in~\cite{BC11}:

\begin{theorem}[{\cite[Theorem 14]{FL18}}]
For $d\geq 1$, the mapping $\Phi_+$ is a bijection
between the family $\mathcal O_d$ and the family $\cB_d$. Each vertex of outdegree $r$ in $\mathcal{O}_d$
becomes a round  vertex of degree $r$ in $\cB_d$, and each non-root face of degree $r$ and ccw-degree $s$
in $\mathcal{O}_d$ becomes a black vertex of degree $r$ with $s$ neighbours (and $r-s$ buds) in $\cB_d$. 
\end{theorem}

\subsection{Specialization to S-quad 3-biorientations}\label{sec:spec_Squad}
We let $\cX$ be the family of S-quad 3-biorientations where the underlying toroidal map is bipartite  
 6-quadrangular, and the biorientation is in $\cO_6$. On the other hand, we let $\cUh$ be the family of 
bimobiles of excess $6$ such that the round vertices are partitioned into black and white vertices without
 two white round vertices or two black round vertices adjacent, where all round vertices have degree $3$
and all square vertices have degree $4$ with $3$ incident buds. Note that $\cU$ and $\cUh$ are easily 
in bijection: indeed each $U\in\cU$ can be turned into a bimobile $\iota(U)\in\cUh$, upon replacing every leaf by a square vertex with $3$ buds
attached (Euler's formula and the fact that $U$ is precubic easily imply that $\iota(U)$ must have excess $6$). 

\begin{proposition}\label{prop:Phi+3biori}
The mapping $\Phi_+$ induces a bijection between $\cX$ and $\cUh$, hence a bijection between $\cX$ and $\cU$. 
The bijection from $\cX$ to $\cU$ amounts to the deletion of the ingoing half-edges.  
\end{proposition}

\begin{figure}[!h]
\center
\includegraphics[scale=0.98]{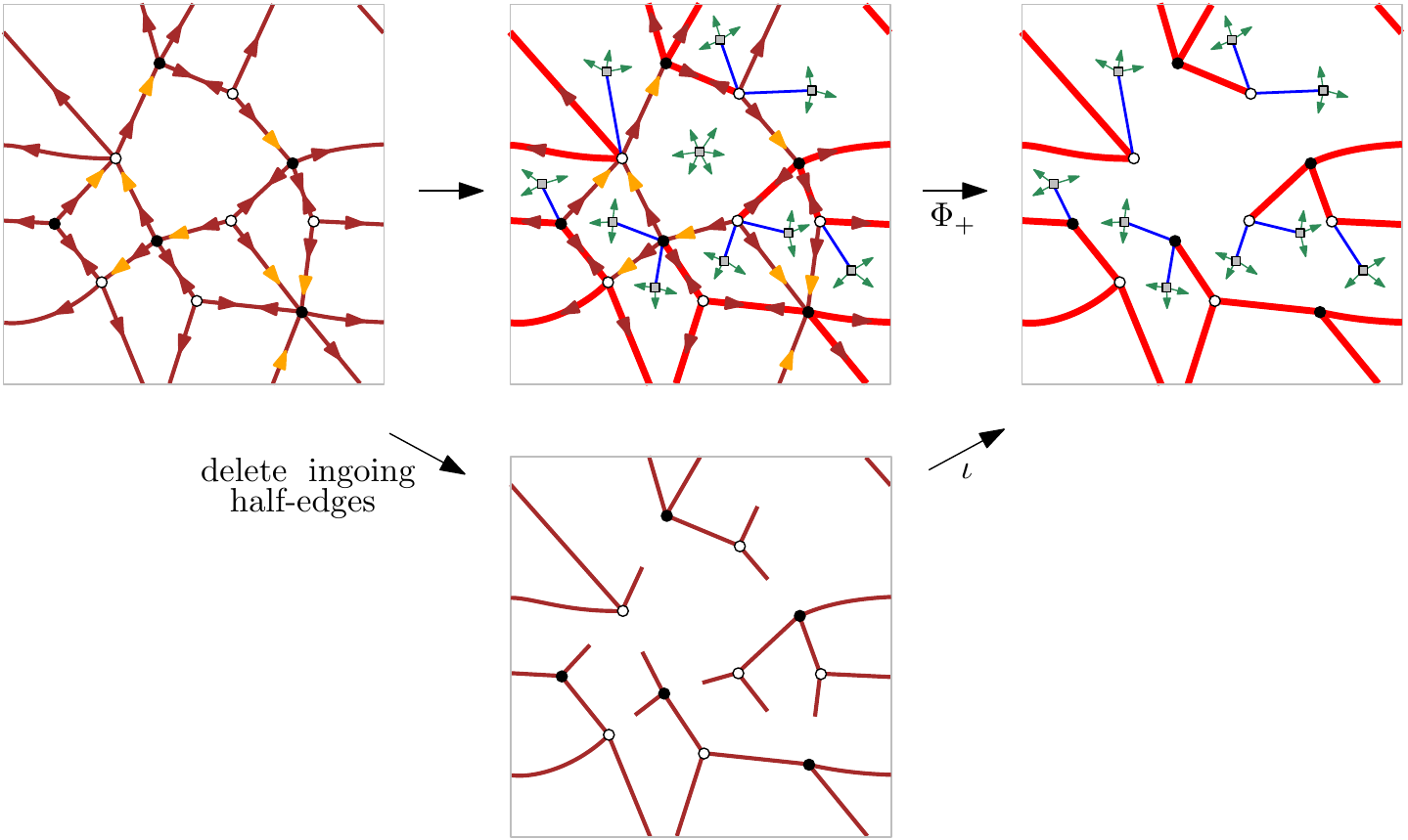}
\caption{The bijection from $\cX$ to $\cU$ amounts to deleting the ingoing half-edges.}
\label{fig:two_versions_phi}
\end{figure}

\begin{proof}
Let $\cX'$ be defined as $\cX$ except that no vertex-bipartition is required. Similarly, let
 $\cUh'$ be defined as $\cUh$ except that no bipartition of the round vertices is required. Clearly
 $\Phi_+$ induces a bijection between $\cX'$ and $\cUh'$. Moreover, if the vertices in $\cX'$ are bipartitioned
(the underlying map is bipartite) then it induces a bipartition of the round vertices in $\cUh$ so that
every round-round edge has a white extremity and a black extremity (indeed such an 
edge already exists in the bidirected map). It remains to show that if the round
vertices of $U\in\cUh'$ can be bipartitioned, then it induces a bipartition of the vertices of the associated
bidirected map $M$. The crucial point is that since $U$ is a toroidal map (with one face) and since the square vertices 
are leaves, the embedded graph $V$ obtained from $U$ by deleting all pending edges is a toroidal map. Hence 
there are two non-homotopic non-contractible cycles in $V$. Since the colors alternate along each cycle, these two cycles have even length. Note also that these two cycles are also present in the bidirected map $M$ (they are cycles of bidirected edges), hence $M$ has two non-homotopic non-contractible cycles of even length. Since the faces of $M$
have even degree, we conclude that $M$ has a valid vertex-bipartition (which is unique up to the choice of the color of a given vertex). Note also that $V$ is a connected spanning submap of $M$. Hence the vertex-bipartition of $V$ is a valid vertex-bipartition for $M$ (such that no two adjacent vertices have the same color), so that $M$ inherits bipartiteness from~$U$. 
Finally the fact that the bijection from $\cX$ to $\cU$ amounts to deleting the ingoing half-edges is 
shown in Figure~\ref{fig:two_versions_phi}.  
\end{proof}

We now give two statements ensuring that Theorem~\ref{theo:bij_phi} follows from Proposition~\ref{prop:Phi+3biori}.

\begin{lemma}\label{lem:necess}
Let $H$ be a toroidal bipartite 6-quadrangular map.
If $H$ can be endowed with a biorientation $X\in\cX$, then $H\in\cH$. 
\end{lemma}
\begin{proof}
Let $W$ be a contractible closed walk of length $2k$ in $H$, such that there is at least one vertex
strictly inside (i.e., not on $W$) the enclosed region $R$. Let $n\geq 1$ and $e$ be respectively the numbers of vertices  and edges strictly inside $R$. The Euler relation ensures that $e=2n+k-\epsilon$ where $\epsilon=2$ if $R$ does not contain the root-face $f_0$, and $\epsilon=3$ if $R$ contains $f_0$. We let $U\in\cUh$ be the mobile associated to $X$,
and let $U_R$ be the part of $U$ restricted to the round vertices strictly inside $R$ and to the square vertices 
for the faces in $R$; in $U_R$ we do not retain the three buds at each square vertex. 
Note that $U_R$ is a forest (if $U_R$ had a cycle,
it would yield a contractible cycle in $U$, a contradiction) where the round vertices have degree $3$ and the square vertices have degree in $\{0,1\}$. It has $n$ nodes, 
 and from the local conditions of $\Phi_+$ one easily sees that the number of edges in $U_R$ is equal
to the number of bidirected edges with both ends strictly inside $R$ plus the number of 
simply directed edges starting from a vertex strictly inside $R$. In particular
$U_R$ has at most $e$ edges. On the other hand, the degree conditions imply that its number of edges is $2n+d$,
with $d\geq 1$ the number of connected components (trees) not reduced to a single square vertex. 
We conclude that $2n+1\leq 2n+d\leq e=2n+k-\epsilon$, hence $k\geq 3$ if $R$ does not contain
$f_0$ and $k\geq 4$ if $R$ contains $f_0$. This easily ensures that $H$ has no contractible 2-cycle (since 
there is no face of degree $2$, such a 2-cycle would have to strictly enclose at least one vertex) nor a non-facial contractible 4-cycle, nor a contractible 6-cycle enclosing $f_0$ and different from the contour of $f_0$.     
\end{proof}

\begin{lemma}\label{lem:suffi}
Every map $H\in\cH$ has a unique biorientation in $\cX$ such that the associated (by Proposition~\ref{prop:Phi+3biori}) 
unicellular map $U\in\cU$ is in $\cUbal$. This biorientation is actually the canonical biorientation of $H$.
\end{lemma}
In order to complete the proof of Theorem~\ref{theo:bij_phi} it just remains to prove Lemma~\ref{lem:suffi}, which we do in the next section.

\subsection{Proof of Lemma~\ref{lem:suffi}}\label{sec:proof_lem_suffi}

\subsubsection{Properties of 3-biorientations}
In this section we let $Q\in\cQ$, and let $M\in \cT$ be the map whose angular map is $Q$.
We let $X$ be an S-quad 3-biorientation of $Q$. For $W$ a closed walk on $Q$ that encloses
a region $R$ homeomorphic to an open disk on its right, we let $cw(W)$ (resp. $ccw(W)$)  
be the number of outgoing half-edges of $W$ that are encountered just after (resp. just before) 
a vertex while walking along $W$ around $R$, and let $o(W)$ be the
number of outgoing half-edges of $X$ that are in the interior of $R$
and incident to a vertex of $W$.

\begin{lemma}
  \label{lem:comptagecycle}
  If $W$ has length $2k$, then  $cw(W)+o(W)=3(k-1)$.
\end{lemma}

\begin{proof}
  Consider the map $G$ obtained from $Q$ by keeping all the
  vertices and edges that lie in the region $R$, including $W$. The
  vertices that appear several times on  $W$ are repeated,
  so $G$ is consider as a planar map whose outer face boundary is $W$ turned into a simple cycle.
In $G$ we call \emph{inner} the vertices and edges not incident to the outer face, and call \emph{inner faces}
those different from the outer face.  
  Let $n,m,f$ be the numbers of inner vertices, inner edges and inner faces of $G$.  
By Euler's formula, $n-m+f=1$. All the inner faces of $G$ have degree $4$ and
  the outer face has degree $2k$, so $4f=2m+2k$. Combining the two
  equalities gives $f=n+k-1$ and $m=2n+k-2$. 

Since $X$ is a 3-biorientation, the total number of outgoing half-edges on inner edges of $G$ is $3n+o(W)$. 
On the other hand, there are $2k-cw(W)$ outer ingoing half-edges with an inner face on their left, and since $X$ is S-quad, there are exactly $f$ ingoing half-edges with an inner face on their left, so the total number of inner ingoing half-edge of $G$ is $f-2k+cw(W)$. Hence $3n+o(W)+f-2k+cw(W)=2m$, so that
$cw(W)+o(W)=2m+2k-f-3n=3k-3$.
  \end{proof}

For $C$ a non-contractible cycle of $Q$ given with a traversal direction, we denote by $o_R(C)$
(resp. $o_{L}(C)$) the total number of half-edges going out of a
vertex on $C$ on the right (resp. left) side of $C$, and let 
$o_{\uparrow}(C)$ (resp. $o_{\downarrow}(C)$) be the total number of
half-edges of $C$ that are outgoing and oriented forward
(resp. backward) along $C$.  Let
$\bar\gamma_R(C)=o_R(C)+o_{\uparrow}(C)$,
$\bar\gamma_L(C)=o_L(C)+o_{\downarrow}(C)$, and 
$\bar\gamma(C)=\bar\gamma_R(C)-\bar\gamma_L(C)$.  

Let $\hM$ be the derived map of $M$, where the primal vertices,
dual vertices, and edge-vertices are respectively white, black and gray. 
 We define the \emph{triangulation} of $\hM$ as the map, denoted $T(\hM)$, obtained
from $\hM$ by inserting a red vertex $v_f$ inside each face $f$, and connecting $v_f$ to the vertices at the $4$ corners around $f$ (see Figure~\ref{fig:triang}, note that $T(\hM)$ is $\hM$ superimposed with the angular map of $\hM$).

\begin{figure}[!h]
\center
\includegraphics[scale=0.8]{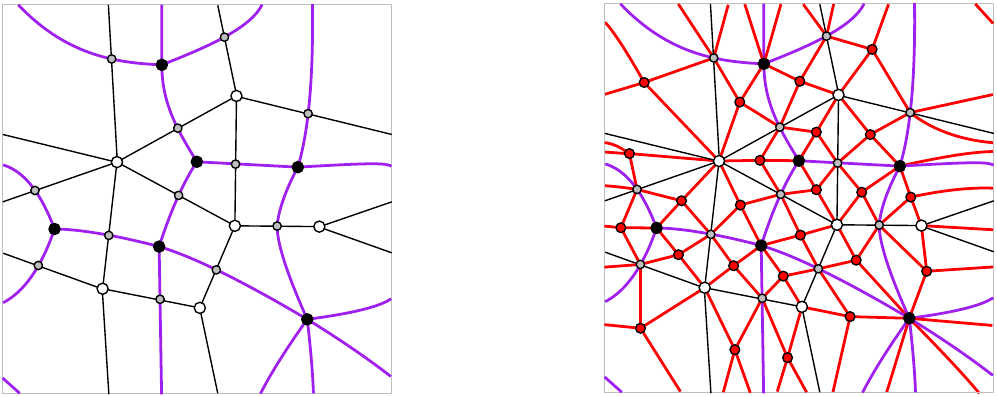}
\caption{Left: the derived map $\hM$ of Figure~\ref{fig:maps}. Right: the associated triangulation $T(\hM)$.}
\label{fig:triang}
\end{figure} 

We recall from~\cite{GL14} that for $T$ a toroidal triangulation, a \emph{3-orientation} of $T$ is 
an orientation where every vertex has outdegree $3$. 
For $M$ an essentially 3-connected toroidal map, let $Y$ be a Schnyder orientation of $\hM$.   
 Then the \emph{induced 3-orientation} of $T(\hM)$ is the orientation $Z$ of $T(\hM)$ obtained from $Y$
by orienting the edges of $T(\hM)\backslash\hM$ according to the rule shown in Figure~\ref{fig:augmented}. 
It is easy to see that it indeed gives a 3-orientation of $T(\hM)$ (red vertices have outdegree $3$,
primal and dual vertices have their outdegree that remains equal to $3$, and edge-vertices have their
outdegree that increases by $2$, from $1$ to $3$). This construction and the following claim will be useful
to prove Lemma~\ref{lem:bal_pseudo_bal} stated next.

\begin{claim}[\cite{GL14,FL18}]\label{claim:triangul}
Let $T$ be a toroidal triangulation endowed with a 3-orientation~$Z$. Then, for $Z$ to be balanced
it is enough that there are two non-homotopic non-contractible cycles that have zero $\gamma$-score. 
\end{claim}

\begin{figure}[!h]
\center
\includegraphics[scale=0.4]{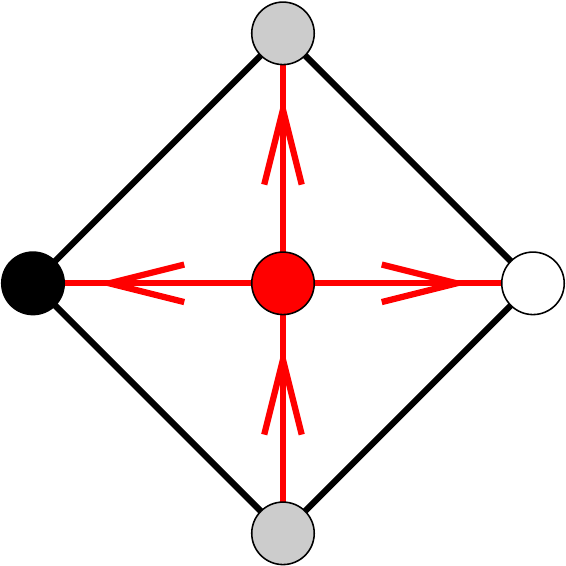}
\caption{The rule to orient the edges of $T(\hM)\backslash\hM$ to obtain the induced 3-orientation.}
\label{fig:augmented}
\end{figure}

\begin{lemma}
  \label{lem:bal_pseudo_bal}
  If $X$ is balanced, then $\bar\gamma(C)=0$ for any non-contractible cycle $C$ of~$Q$. 
Moreover, for $X$ to be balanced it is enough that $\bar\gamma(C_1)=\bar\gamma(C_2)=0$ for a pair $\{C_1,C_2\}$ 
of non-homotopic non-contractible cycles of $Q$. 
\end{lemma}
\begin{proof}
Let $C$ be a non-contractible cycle of $Q$ given with a traversal direction, and let $Z$ be the 3-orientation of $T(\hM)$ induced by $Y=\sigma(X)$. Note that the map obtained by 2-subdividing every edge of $Q$ is a submap of $T(\hM)$, hence every cycle $C$ of $Q$ identifies to a cycle of $T(\hM)$, which we denote by $\C2$. Let $2k$ be the length of $C$. The local rule (Figure~\ref{fig:augmented}) to obtain $Z$ out of $Y$ ensures that there is a total contribution of $k$ from the red vertices to
 $\gamma_L^Z(\C2)$ (resp. to $\gamma_R^Z(\C2)$). On the other hand, the local rules (Figure~\ref{fig:transfer}) to obtain $Y$ from $X$
ensure that  there is a contribution of $\bar\gamma_L^X(C)$ (resp. $\bar\gamma_R^X(C)$) from the black and white vertices
to $\gamma_L^Z(\C2)$ (resp. $\gamma_R^Z(\C2)$). Hence we have 
\[
\bar\gamma^X(C)=\gamma^Z(\C2).
\]
On the other hand, if $C$ is a non-contractible cycle of $\hM$, then  $C$  is also a non-contractible cycle of $T(\hM)$ (indeed $\hM$ is a submap of $T(\hM)$). Let $C$ be a non-contractible cycle of $\hM$ not passing by dual vertices. 
Let $2k$ be the length of $C$ (which alternates between primal vertices and edge-vertices). 
Then the local rule to obtain $Z$ from $Y$ 
ensures that $\gamma_L^Z(C)=\gamma_L^Y(C)+k$ and  $\gamma_R^Z(C)=\gamma_R^Y(C)+k$, hence
\[
\gamma^Z(C)=\gamma^Y(C).
\]
We can now easily conclude the proof. By definition $X$ is balanced iff $Y$ is balanced, iff $\gamma^Y(C)=0$
for every non-contractible cycle $C$ of $\hM$ not passing by dual vertices. We can take a pair $\{C_1,C_2\}$ 
of such cycles that are non-homotopic.   
We have $\gamma^Z(C_1)=\gamma^Z(C_2)=0$, which ensures that $Z$ is balanced according to Claim~\ref{claim:triangul}. 
Hence, for any non-contractible cycle $C$ of $Q$ we have $\bar\gamma^X(C)=\gamma^Z(\C2)=0$. 

Conversely, if there is a pair $\{C_1,C_2\}$ 
of non-homotopic non-contractible cycles of $Q$ such that $\bar\gamma^X(C_1)=\bar\gamma^X(C_2)=0$, 
then $\gamma^Z(C_1^{(2)})=\gamma^Z(C_2^{(2)})=0$, hence $Z$ is balanced. Hence for any non-contractible cycle $C$ 
of $\hM$ not passing by dual vertices, we have $\gamma^Y(C)=\gamma^Z(C)=0$. Hence
$Y$ is balanced, and so is $X$. 
\end{proof}

\begin{lemma}
  \label{lem:rightmost}
 Assume $X$ is balanced. Then any rightmost walk of $X$ eventually loops on a closed walk $W$ of length $6$
with a contractible region on its right.
\end{lemma}
\begin{proof}
We may assume that all the half-edges
  of $W$ are distinct, i.e., there is no strict subwalk of $W$ that is
  also a looping part of a rightmost walk.  Note that $W$ cannot cross
  itself otherwise it is not a rightmost walk. However, $W$ may have
  repeated vertices but in that case $W$ intersects itself
  tangentially on the left side. Let $2k$ be the length of $W$.

  Suppose by contradiction that there is an oriented subwalk $\bar{W}$ of
  $W$ (possibly $\bar{W}=W$) that encloses a region $R$ on its left side
  that is homeomorphic to an open disk. We let $W'$ be $\bar{W}$ taken in reverse direction, so that
it encloses $R$ on its right side.   Let $v$ be the vertex
  incident to the first half-edge of $W'$ (from an arbitrary starting point on the oriented walk $W'$). Let $2k'$ be the length of
  $W'$.  By Lemma~\ref{lem:comptagecycle}, we have
  $cw(W')+o(W')=3(k'-1)$. Note that $ccw(W')=2k'$. So the number of
  outgoing half-edges of the $2k'$ vertices of $W'$ in $R$ (including
  its border) is $5k'-3$. Since $W$ is a rightmost walk, all the $2k'-1$
  vertices of $W'$ distinct from $v$ have their $3$ incident outgoing
  half-edges in $R$. Note that $v$ has at least one outgoing half-edge
  in $R$ (the first half-edge of $W'$). So the number of outgoing
  half-edges incident to vertices of $W'$ that lie in $R$ is at least
  $3\times (2k'-1)+1$.  Thus $5k'-3\geq 3\times (2k'-1)+1$, a
  contradiction.

Let $k$ be the half-length of $W$. We consider two cases whether $W$ is a cycle
  (no repetition of vertices) or not.

\begin{figure}[!h]
 \center
 \begin{tabular}{cc}
 \includegraphics[scale=0.3]{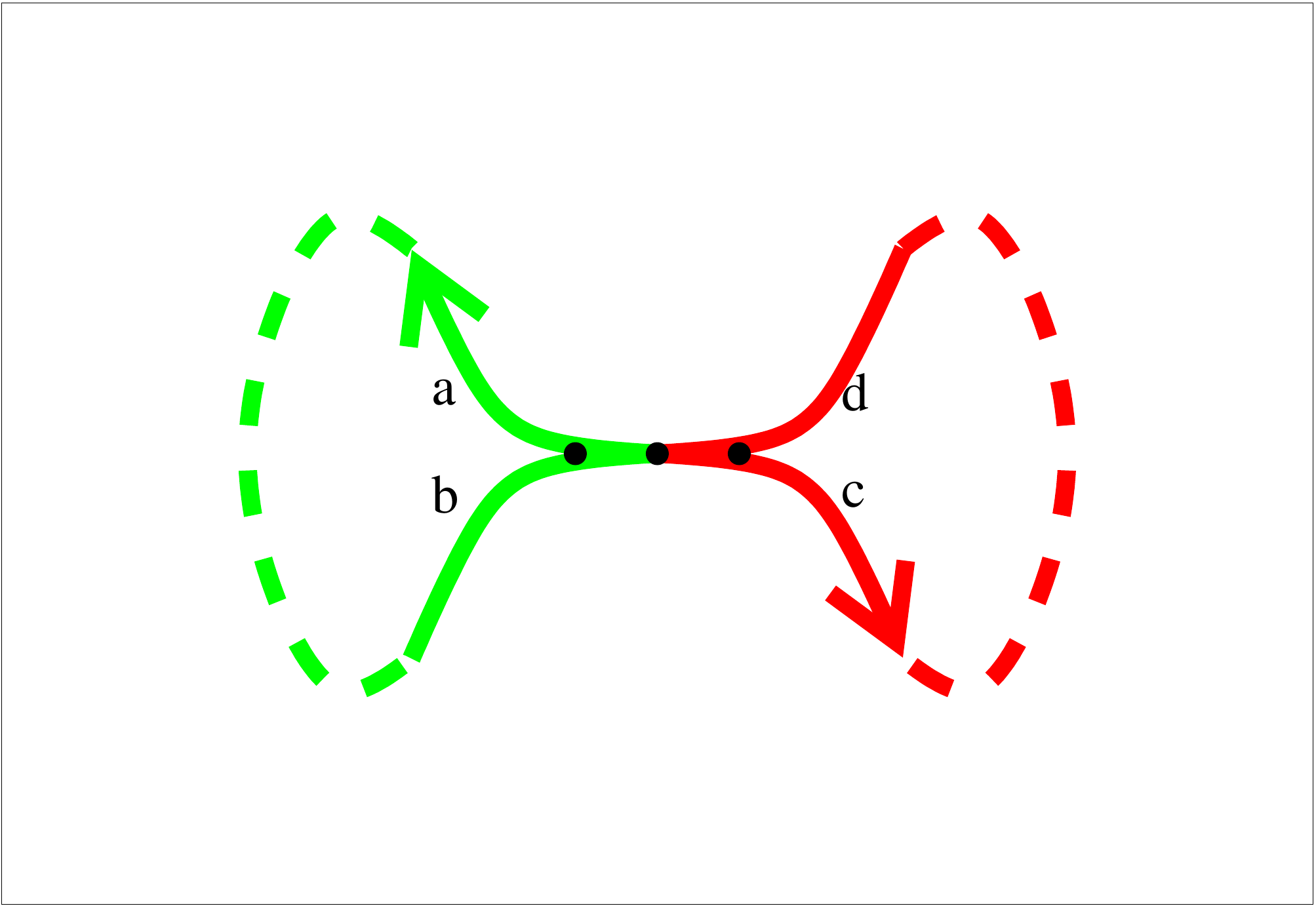} \ \ & \ \ 
 \includegraphics[scale=0.3]{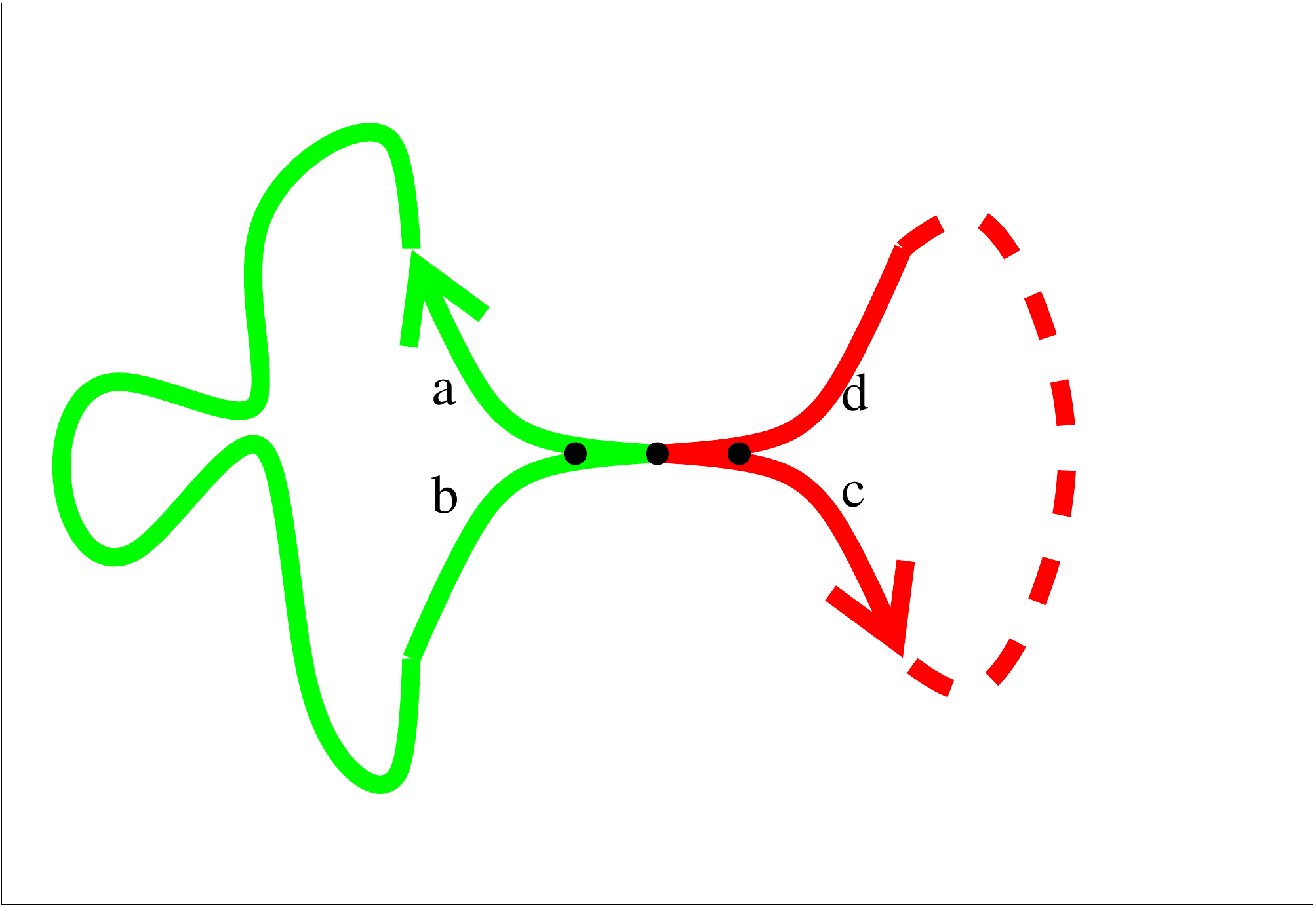}\\
(i) \ \ & \ \ (ii)\\
& \\
 \includegraphics[scale=0.3]{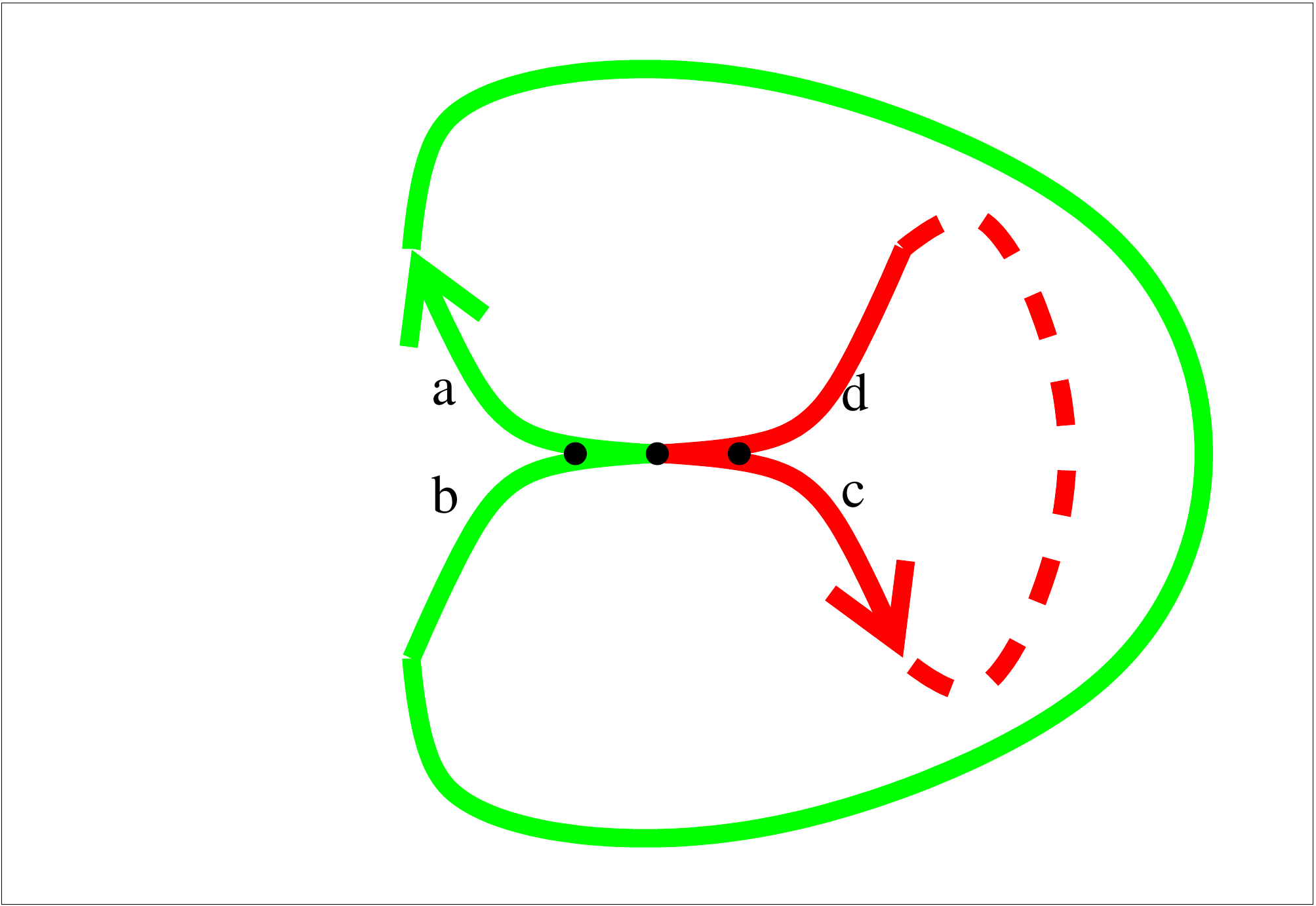}\ \ & \ \ 
 \includegraphics[scale=0.3]{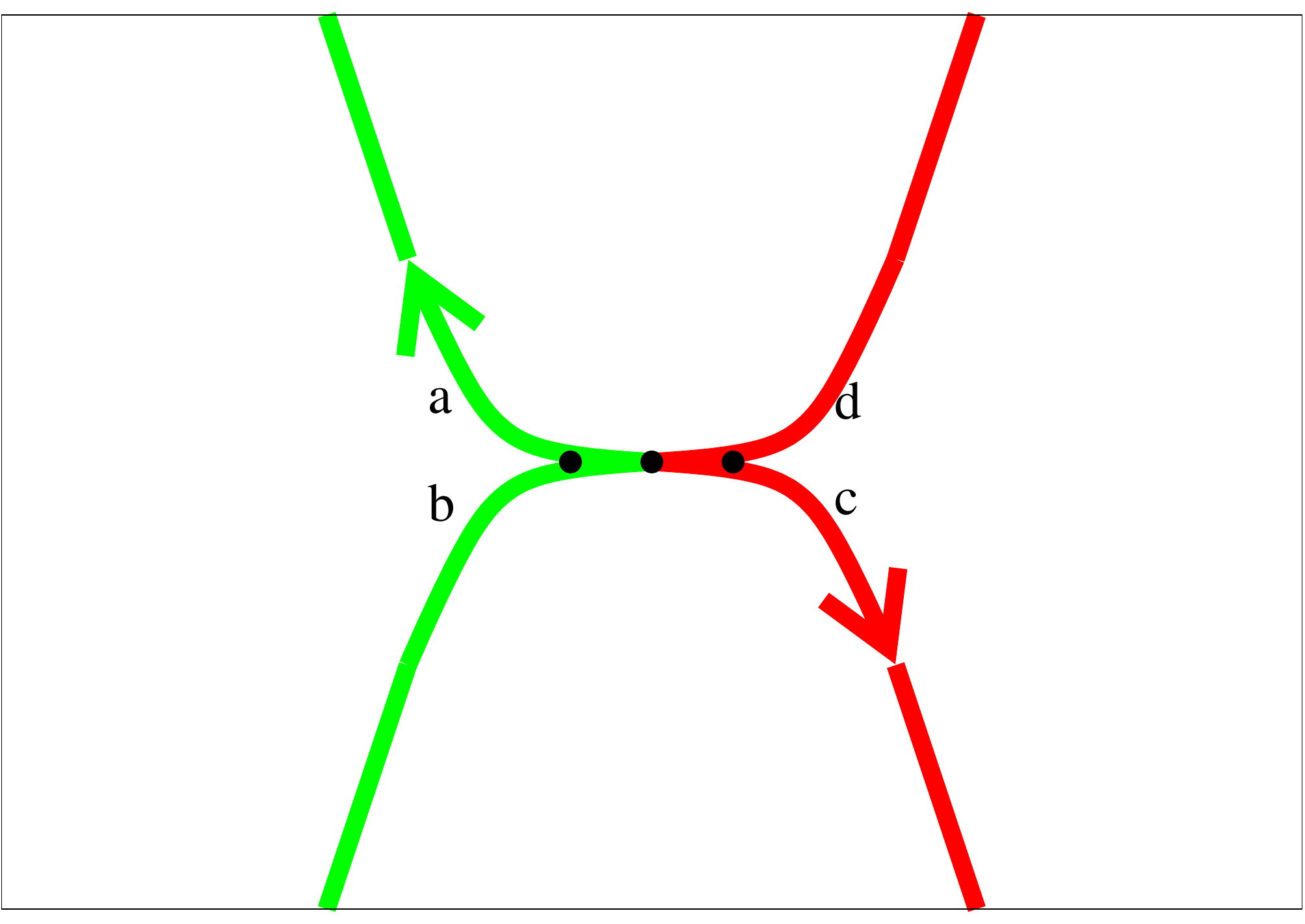}\\
(iii) \ \ & \ \ (iv)\\
 \end{tabular}
\caption{Case analysis for the proof.}
 \label{fig:righttangent}
 \end{figure}

\begin{itemize}
\item \emph{$W$ is a cycle} 

  Suppose by contradiction that $W$ is a non-contractible cycle $C$.
  Since $W$ is a rightmost walk, all the half-edges incident to the
  right side of $C$ are ingoing, so $w_R(C)=0$. We have
  $o_{\uparrow}(C)=2k$ since $W$ is taking an outgoing half-edge 
  after each encountered vertex. Hence $\bar\gamma_R(C)=2k$. Since we are considering a
  $3$-biorientation of $M$, the total number of outgoing half-edges of
  $X$ incident to vertices of $W$ is $3\times 2k$, hence $\bar\gamma_L(C)+\bar\gamma_R(C)=6k$, 
so $\bar\gamma_L(C)=4k$. We reach the contradiction that $\bar\gamma(C)=\bar\gamma_R(C)-\bar\gamma_L(C)=-2k\neq 0$. 
  Thus $W$ is a contractible cycle.

  By previous arguments, the contractible cycle $W$ does not enclose a
  region homeomorphic to an open disk on its left side. So $W$
  encloses a region homeomorphic to an open disk on its right side, as
  claimed.

\item \emph{$W$ is not a cycle} 

  As we have already seen, $W$ cannot cross itself nor intersect itself tangentially on
  the right side, hence it has to intersect tangentially on the left side.
  Such an intersection can be on a single vertex or on a path, as
  illustrated in Figure~\ref{fig:righttangent}(i). The half-edges of $W$
  incident to this intersection are noted as in Figure~\ref{fig:righttangent}(i)--(iv),
  where $W$ is going periodically through (in this order) $a,b,c,d$.  By
  what precedes, the subwalk of $W$ from $a$ to $b$ (shown in green) does
  not enclose regions homeomorphic to  open disks on its left
  side. Hence we are not in the case of 
  Figure~\ref{fig:righttangent}(ii). Moreover, if this (green) subwalk
  encloses a contractible region on its right side,
  then this region contains the (red) subwalk of $W$ from $c$ to $d$,
  see Figure~\ref{fig:righttangent}(iii). Since $W$ cannot cross
  itself, this (red) subwalk necessarily encloses contractible regions
   on its left side, a contradiction. So
  the (green) subwalk of $W$ starting from $a$ must form a
  non-contractible curve before reaching $b$. Similarly, the (red)
  subwalk starting from $c$  must form a
  non-contractible curve before reaching $d$. Since $W$ is a rightmost
  walk  and cannot cross itself, we are in the situation of
  Figure~\ref{fig:righttangent}(iv) (possibly with more tangent
  intersections on the left side). Hence, $W$
  encloses a contractible region on its right side.
\end{itemize}
We conclude that in both cases $W$ encloses a contractible region on its right side. 
Since $W$ is a rightmost walk we have $cw(W)=2k$ and $o(W)=0$. Hence by Lemma~\ref{lem:comptagecycle} we have
$2k=3k-3$, so that $k=3$.  
\end{proof}

\subsubsection{Proof of existence in Lemma~\ref{lem:suffi}}\label{sec:suffi} 
Let $H\in\cH$ and let $Q\in\cQ$ be obtained from $H$ by adding a black vertex $v_0$ inside the 
root-face $f_0$ of $H$, and adding 3 edges $e_1,e_2,e_3$ connecting $v_0$ to the corners  $c_1,c_2,c_3$
at white vertices around $f_0$. Let $M$ be the map whose angular map is $Q$, $\hM$ the derived map of $M$, 
 and let $f_1,f_2,f_3$ be the 
$3$ faces of $\hM$ corresponding to $e_1,e_2,e_3$. 
Let $Y$ be the minimal balanced Schnyder orientation of $\hM$ (minimal with respect to any of the 
$3$ faces $f_1,f_2,f_3$), let $X'=\sigma^{-1}(Y)$, and let $X$ be the induced biorientation of $H$ (the canonical biorientation of $H$).    
Let $W_0$ be the closed walk of length $6$ in $Q$ that is inherited from the contour of $f_0$,
and let $R_0$ be the contractible region of $Q$ enclosed by $W_0$ (i.e., $R_0$ is 
the $3$-face region that takes the place of $f_0$). 

\begin{lemma}\label{lem:inO6}
In $Q$, any rightmost walk for $X'$ eventually loops on $W_0$. This implies that $X$ is an S-quad
 3-biorientation in $\cO_6$. 
\end{lemma}
\begin{proof}
Let $P$ be a rightmost walk in $Q$ and let $W$ be the eventual looping part of~$P$. Lemma~\ref{lem:rightmost} ensures that $W$ has length $6$ and has a contractible region $R$ on its right. 
As shown in Figure~\ref{fig:interior}, within $R$ there has to be a clockwise closed walk of $Y$. Since $Y$ is minimal, the region enclosed by this walk 
has to contain $f_1$ (we choose $f_1$, the same would work with $f_2$ or $f_3$). Hence $R$ contains $e_1$,
so that $W_0$ is an enclosing hexagon for $Q$ where $e_1$ is taken as the root-edge.  
From Claim~\ref{claim:decomp_HQ} 
we know that $W_0$ is the maximal hexagon enclosing $e_1$, and it is directly checked that $W_0$ is 
the unique hexagon enclosing~$e_1$ (indeed within $W_0$ there is just $v_0$ and its $3$ incident edges). 
Hence $W=W_0$. 

\begin{figure}[!h]
\center
\includegraphics[scale=0.4]{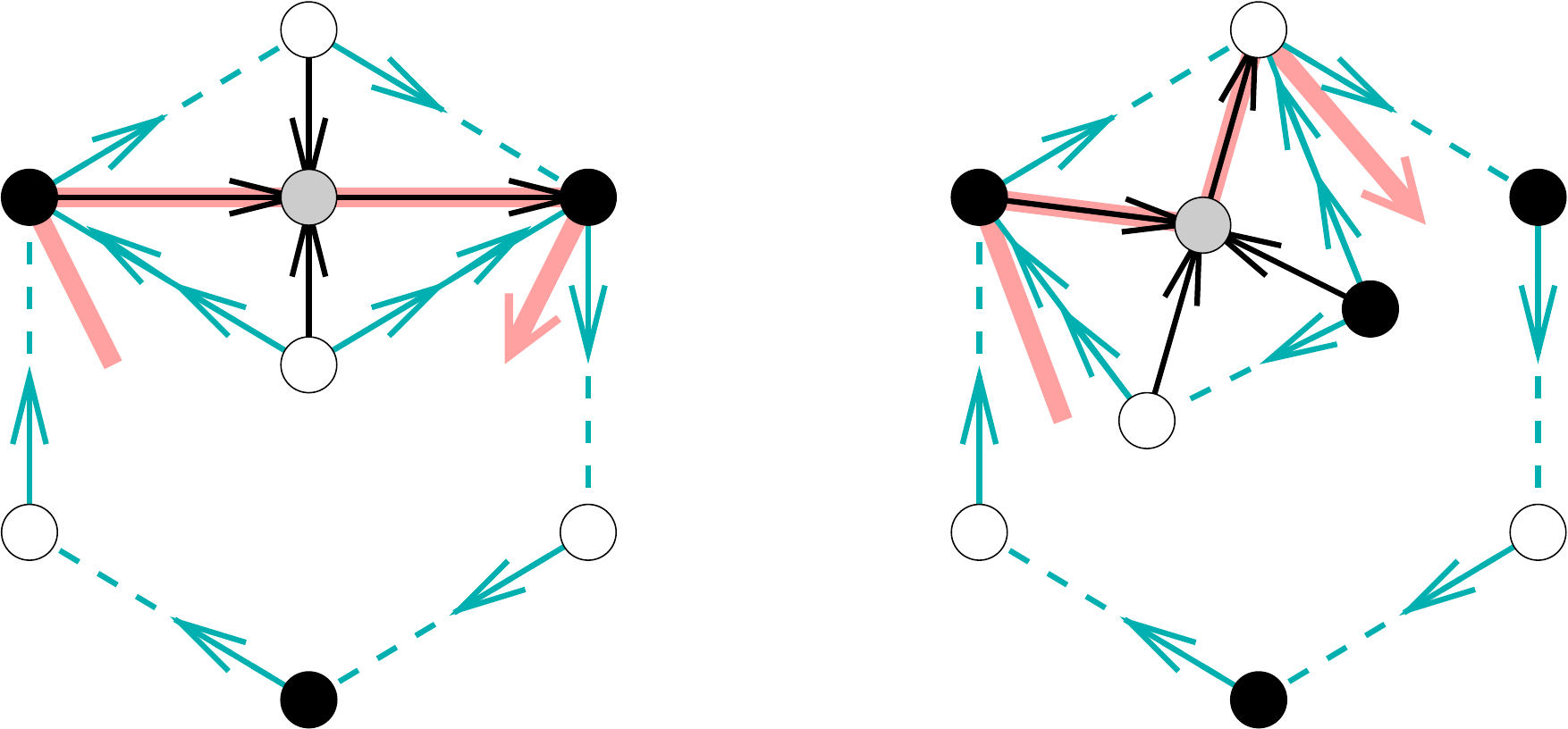}
\caption{Existence of a clockwise closed walk of $Y$ within the eventual looping part (forming a hexagon) of a rightmost walk of $X$, in the proof of
 Lemma~\ref{lem:inO6}.}
\label{fig:interior}
\end{figure}

This also implies that $W_0$ is a rightmost walk, so that $o(W_0)=0$
by Lemma~\ref{lem:comptagecycle}, hence the edges $e_1,e_2,e_3$ are 
simply directed out of $v_0$. Hence $X$ is an S-quad 3-biorientation, and it is in $\cO_6$. 
\end{proof}

\begin{lemma}
Let $U\in\cU$ be the unicellular map associated to $X$ (i.e., obtained from $X$ by deleting the ingoing half-edges).
 Then $U\in\cUbal$. 
\end{lemma}
\begin{proof}
Note that any non-contractible cycle $C$ of $U$
is a non-contractible cycle of bidirected edges in $X$. Since $X$ is balanced we have $\bar\gamma(C)=0$. 
Since all the edges on $C$ are bidirected, this implies that $C$ has as many incident outgoing half-edges on both sides (in $X$). Hence, in $U$, the cycle $C$ has as many incident half-edges on both sides. Hence $U\in\cUbal$.  
\end{proof}

\subsubsection{Proof of uniqueness in Lemma~\ref{lem:suffi}} 
Let $H\in \cH$ and let $X$ be an S-quad 3-biorientation of $H$ in $\cO_6$, and such that the 
unicellular map $U$ obtained from $X$ by deleting the ingoing half-edges is in $\cUbal$. 
We are going to show that $X$ has to be the canonical biorientation of $H$. 
Let $Q$ be obtained from $H$ by adding a black vertex $v_0$ within the hexangular face, and connecting it to the vertices at the 3 white corners. Let $M$ be the map whose angular map is $Q$, and let $\hM$ be the derived map of $M$. Let $X'$ be the S-quad 3-biorientation of $Q$ that is the same as $X$, with the $3$ additional 
edges simply directed out of $v_0$. Let $C_1,C_2$ be two non-homotopic non-contractible cycles of $U$. Since 
$U\in\cUbal$, for $C\in\{C_1,C_2\}$ the cycle $C$ has as many incident half-edges on both sides. Since $U$
is obtained from $X$ by deleting ingoing half-edges, the cycle $C$ is made of bidirected edges and 
has as many incident outgoing half-edges on both sides, hence $\bar\gamma(C)=0$.
 By Lemma~\ref{lem:bal_pseudo_bal} this ensures that $X'$ is balanced.

Let $Y=\sigma(X')$ be the associated Schnyder orientation of $\hM$. Note that $Y$ is balanced, since by definition $X'$ is balanced
iff $Y$ is balanced. Let $f_0$ be one of the 3 faces of $\hM$ incident to $v_0$, which we take as the root-face 
of $\hM$.  
In order to prove that $X$ is the canonical biorientation of $H$, it just remains to show that $Y$ is minimal
 with respect to~$f_0$. 

\begin{figure}
\center
\includegraphics[scale=0.4]{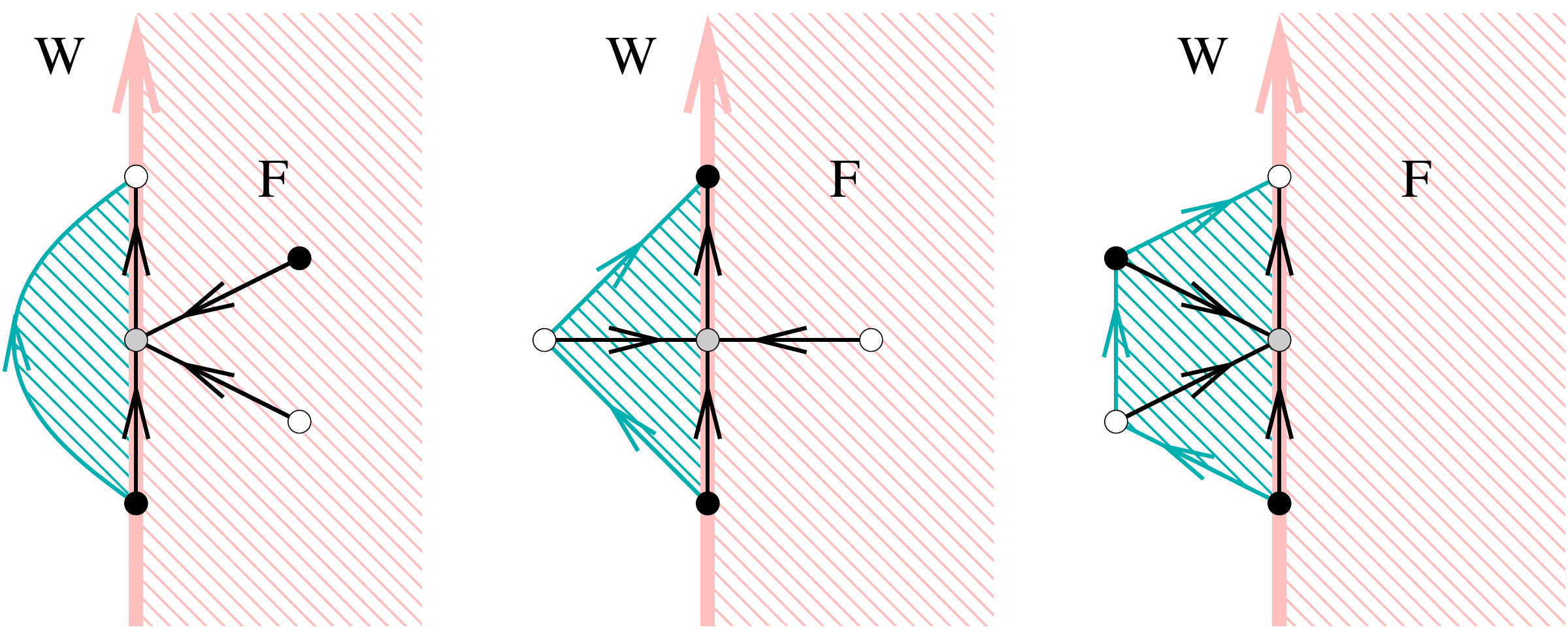}
\caption{Orientation of the edges of the expansion, on the left side of an oriented walk of the derived map.}
\label{fig:nonminimal}
\end{figure}

Suppose by contradiction that $Y$ is non-minimal, i.e., there
exists a non-empty set $F$ of faces such that $f_0\notin F$ and every
edge on the boundary $\partial F$ of $F$ has a face in $F$ on its right (and a
face not in $F$ on its left). 
Walking on $\partial F$, we have an alternation of 
 vertices that are not edge-vertices (i.e., primal or dual vertices) and edges-vertices. 
One can easily see that each pair of consecutive oriented edges on $\partial F$, located before and
after an edge-vertex, corresponds to an oriented path of length
1, 2, or 3 of $Q$ on the left side of $\partial F$, see
Figure~\ref{fig:nonminimal}, where the first vertex of the two
consecutive edges of $W$ is a dual vertex (the cases where this
vertex is a primal vertex are completely analogous). 
 Let $F'$ be the set of faces of $Q$ corresponding to edge-vertices of $Y$ that are incident
to a face of $F$.  
Note that $F'$ is the union of the 
region delimited by $F$, depicted in red on
Figure~\ref{fig:nonminimal}, plus the union of all the regions
depicted in green on
Figure~\ref{fig:nonminimal}. Hence every edge on the boundary of $F'$ (on the left boundary of a green region) 
is either bidirected or 
simply directed with $F'$ on its right. Since $f_0\notin F$ and $v_0$ is a source in $Y$, the vertex $v_0$ can not
be incident to any face in $F'$, hence $F'$ avoids the 3 faces of $Q$ incident to $v_0$, 
and thus $F'$ can be seen as a set of faces of $H$ that does not contain the hexangular face. Hence, if we let $\tilde{X}$ 
be the orientation obtained from $X$ by turning every bidirected edge into a double edge forming a clockwise cycle
(enclosing a face of degree $2$), 
then $\tilde{X}$ is not minimal (with respect to the hexagonal face).  
By Lemma~16 in~\cite{FL18} we conclude that $X\notin\cO_6$, a contradiction.  

\section{Proof that $\phi$ is the inverse of $\psi$}\label{sec:proof_inverse} 
Since we know that $\phi$ is a bijection from $\cH$ to $\cUbal$, showing that $\psi$ is its inverse
amounts to showing that for $U\in\cUbal$, we have that $H:=\psi(U)$ is in $\cH$ and $\phi(\psi(U))=U$. 

\begin{figure}
 \center
 \includegraphics[scale=0.8]{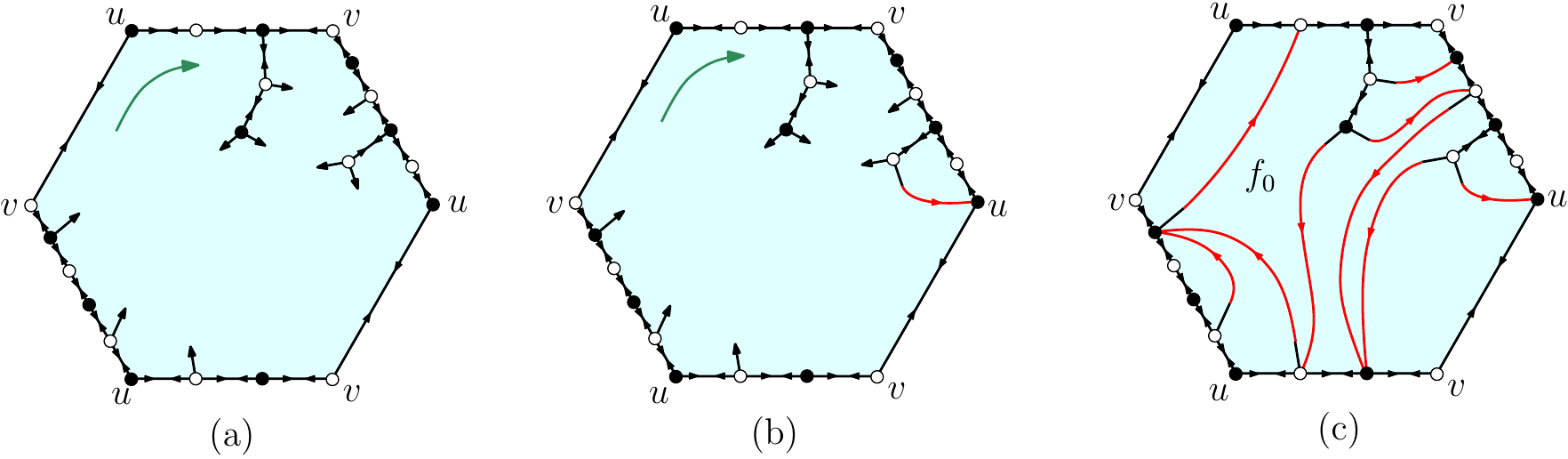}
 \caption{The closure-mapping and the induced biorientation.}
 \label{fig:closure2}
 \end{figure}

Similarly as in the planar case~\cite{FuPoScL}, we note that $H$ inherits from $U$ a  biorientation~$X$. 
Precisely, in $U$ we biorient every plain edge, and we simply orient every pending edge  
toward its incident leaf (which gets merged in a local closure operation, see Section~\ref{sec:bij_psi}), 
see Figure~\ref{fig:closure2}.

Let $m$ be the number of leaves of $U$. We perform the $m$ local closures (one for each leaf) 
in a given (arbitrary) order. For $t\in[0..m]$ we let 
$X_t$ be the bioriented map obtained from the figure after $t$ local closures, where the $(m-t)$ pending edges    
are deleted. It is easy to see that, for $t$ from $0$ to $m$, the root-face contour of $X_t$ remains a rightmost walk (its ccw-degree remains zero) and that $X_t\in\cO_{2(m-t)+6}$. Hence $X:=X_m$ is in $\cO_6$. It is also clear that the final outdegree of each vertex is the same as its degree in $U$, hence is $3$, and each closed quadrangular face gets ccw-degree $1$. Moreover, $H$ is bipartite and 6-quadrangular by construction. Hence
$X\in\cX$. This guarantees that $H\in\cH$, by Lemma~\ref{lem:necess}.    
Note also that $U$ is recovered from $X$ by deleting the ingoing half-edges (which undoes the closure-operations). 
Since $U\in\cUbal$, we conclude from Lemma~\ref{lem:suffi} that $X$ is the canonical biorientation 
of $H$ and that $\phi(H)=U$.

\bibliographystyle{plain}
\bibliography{bibli}

\end{document}